\documentclass[11pt,reqno]{amsart} 
\pdfoutput=1 

\usepackage[mathlines]{lineno}
\usepackage{microtype}

\usepackage[top=1.25in,bottom=1.25in,left=1.25in,right=1.25in]{geometry} 

\usepackage{xcolor}

\definecolor{my-red}{rgb}{0.5,0.0,0.0}
\definecolor{my-blue}{rgb}{0.0,0.0,0.6}
\definecolor{my-green}{rgb}{0.0,0.5,0.0}
\definecolor{light-gray}{gray}{0.6}

\usepackage{enumitem}

\usepackage{tocvsec2}

\usepackage[labelfont=bf]{caption}

\usepackage[font={small},labelfont=md]{subfig}

\usepackage[noadjust,nosort]{cite}

\usepackage{amsmath, amsthm, amssymb, tikz, mathtools, array, mathrsfs, tensor, ifthen, xparse, graphicx}

\usepackage[foot]{amsaddr}

\usepackage{polynom}

\usepackage{esint}

\usepackage{multicol}

\usepackage{dsfont}
	\newcommand{\one}{\mathds{1}}

\usepackage{framed}

\usepackage{scalerel}
\newcommand{\sig}{{\scaleobj{1.1}{\diamond}}} 

\usepackage{stackengine}
	
	\newcommand\Barabove[1]{\overset{\rule{2ex}{.075ex}}{#1}}
    \newcommand\barabove[1]{\overset{\rule{1.3ex}{.075ex}}{#1}}

\usepackage{upref}

\usepackage[pagebackref=true, colorlinks=true, urlcolor=light-gray, linkcolor=my-blue, citecolor=my-green]{hyperref}

\usepackage{orcidlink}

\allowdisplaybreaks

\numberwithin{equation}{section}


\newcommand{\graybrace}[2]{
	\begingroup
	\color{gray}
	\underbrace{\color{black}#1}_{\substack{#2}}
	\endgroup
}

\newcommand{\eq}[1]{\begin{linenomath}\postdisplaypenalty=0\begin{align*} #1 \end{align*}\end{linenomath}}

\NewDocumentCommand{\eeq}{om}{\begin{linenomath}\postdisplaypenalty=0\begin{align} \IfNoValueTF{#1}{}{\tag{#1}} \begin{split} #2 \end{split} \end{align}\end{linenomath}}

\newcommand{\stackref}[2]{
\readlist*\mylist{#1}
\stackrel{\mbox{\footnotesize\foreachitem\x\in\mylist[]{\ifnum\xcnt=1\else,\fi\eqref{\x}}}}{#2}
}
\newcommand{\stackrefp}[2]{
\readlist*\mylist{#1}
\stackrel{\hphantom{\mbox{\footnotesize\foreachitem\x\in\mylist[]{\ifnum\xcnt=1\else,\fi\eqref{\x}}}}}{#2}
}
\newcommand{\stackrefpp}[3]{
\readlist*\mylist{#1}
\readlist*\mylistt{#2}
\stackrel{\parbox{\widthof{\footnotesize\foreachitem\x\in\mylistt[]{\ifnum\xcnt=1\else,\fi\eqref{\x}}}}{\centering\footnotesize\foreachitem\x\in\mylist[]{{\ifnum\xcnt=1\else,\fi\eqref{\x}}}}}{#3}
}

\def\eps{\varepsilon}

\newcommand{\E}{\mathbb{E}}

\renewcommand{\P}{\mathbb{P}}

\newcommand{\R}{\mathbb{R}}
\renewcommand{\S}{\mathbb{S}}

\newcommand{\Z}{\mathbb{Z}}

\newcommand{\CC}{\mathcal{C}}

\newcommand{\HH}{\mathcal{H}}
\newcommand{\II}{\mathcal{I}}

\newcommand{\ZZ}{\mathcal{Z}}

\newcommand{\PPP}{\mathscr{P}}

\newcommand{\SSS}{\mathscr{S}}

\newcommand*{\QED}{\hfill\ensuremath{\square}}

\newcommand{\iprod}[3][]{#1\langle #2,\, #3 #1\rangle}

\newcommand{\vc}[1]{{\boldsymbol #1}}

\newcommand{\bx}{\boldsymbol{x}}
\newcommand{\bq}{\boldsymbol{q}}

\newcommand{\bone}{\boldsymbol{1}}
\newcommand{\la}{\lambda}
\newcommand{\sx}{\sig}
\newcommand{\wt}[1]{\widetilde{#1}}
\newcommand{\wh}[1]{\widehat{#1}}

\DeclareMathOperator{\GSE}{GSE}
\DeclareMathOperator{\e}{e} 

\newcommand{\cc}{\mathsf{c}} 
\newcommand{\dd}{\mathrm{d}} 

            \makeatletter
            \DeclareFontFamily{OMX}{MnSymbolE}{}
            \DeclareSymbolFont{MnLargeSymbols}{OMX}{MnSymbolE}{m}{n}
            \SetSymbolFont{MnLargeSymbols}{bold}{OMX}{MnSymbolE}{b}{n}
            \DeclareFontShape{OMX}{MnSymbolE}{m}{n}{
                <-6>  MnSymbolE5
               <6-7>  MnSymbolE6
               <7-8>  MnSymbolE7
               <8-9>  MnSymbolE8
               <9-10> MnSymbolE9
              <10-12> MnSymbolE10
              <12->   MnSymbolE12
            }{}
            \DeclareFontShape{OMX}{MnSymbolE}{b}{n}{
                <-6>  MnSymbolE-Bold5
               <6-7>  MnSymbolE-Bold6
               <7-8>  MnSymbolE-Bold7
               <8-9>  MnSymbolE-Bold8
               <9-10> MnSymbolE-Bold9
              <10-12> MnSymbolE-Bold10
              <12->   MnSymbolE-Bold12
            }{}
            
            \let\llangle\@undefined
            \let\rrangle\@undefined
            \DeclareMathDelimiter{\llangle}{\mathopen}%
                                 {MnLargeSymbols}{'164}{MnLargeSymbols}{'164}
            \DeclareMathDelimiter{\rrangle}{\mathclose}%
                                 {MnLargeSymbols}{'171}{MnLargeSymbols}{'171}
            \makeatother
            
    \DeclareFontFamily{U}{matha}{\hyphenchar\font45}
    \DeclareFontShape{U}{matha}{m}{n}{ <-6> matha5 <6-7> matha6 <7-8>
    matha7 <8-9> matha8 <9-10> matha9 <10-12> matha10 <12-> matha12 }{}
    \DeclareSymbolFont{matha}{U}{matha}{m}{n}
    \DeclareFontFamily{U}{mathx}{\hyphenchar\font45}
    \DeclareFontShape{U}{mathx}{m}{n}{ <-6> mathx5 <6-7> mathx6 <7-8>
    mathx7 <8-9> mathx8 <9-10> mathx9 <10-12> mathx10 <12-> mathx12 }{}
    \DeclareSymbolFont{mathx}{U}{mathx}{m}{n}
    
    \DeclareMathDelimiter{\llbrack} {4}{matha}{"76}{mathx}{"30}
    \DeclareMathDelimiter{\rrbrack} {5}{matha}{"77}{mathx}{"38}

\renewcommand{\le}{\leqslant}
\renewcommand{\leq}{\leqslant}
\renewcommand{\ge}{\geqslant}
\renewcommand{\geq}{\geqslant}
            
\newcommand{\single}{\mathsf{single}}
\newcommand{\prt}{\pi}

\newcommand{\inj}{\mathrm{inj}}
\newcommand{\is}{\mathsf{Ising}}
\newcommand{\sph}{\mathsf{spherical}}
\newcommand{\bbeta}{\boldsymbol{\beta}}

\newtheorem{theorem}{Theorem}[section]
\newtheorem{proposition}[theorem]{Proposition}
\newtheorem{corollary}[theorem]{Corollary}
\newtheorem{lemma}[theorem]{Lemma}

\theoremstyle{definition} 
\newtheorem{definition}[theorem]{Definition}
\newtheorem{example}[theorem]{Example}

\newtheorem{remark}[theorem]{Remark}

\newenvironment{proofclaim}[1][Proof]
	{\begin{proof}[#1]}
	{\end{proof}}

\title[Balanced multi-species spin glasses]{Balanced multi-species spin glasses}



\author[E. Bates]{Erik Bates$^*$\,\orcidlink{0000-0002-3472-036X}}
\address{$^*$Department of Mathematics, North Carolina State University}
\email{ebates@ncsu.edu}

\author[Y. Sohn]{Youngtak Sohn$^\dagger$\,\orcidlink{0009-0009-0038-1417}}

\address{$^\dagger$Division of Applied Mathematics, Brown University} 
\email{youngtak\_sohn@brown.edu}

\begin{document}


\begin{abstract}
We identify a special class of multi-species spin glass models: ones in which the species proportions serve to ``balance'' out the interaction strengths.
For this class, we prove a free energy lower bound that does not require any convexity assumption, and applies to both Ising and spherical models.
The lower bound is the free energy of a single-species model whose $p$-spin inverse-temperature parameter is exactly the variance of the $p$-spin component of the multi-species Hamiltonian.
For the Ising case, this generalizes an inequality recently found by Issa~\cite{issa24_arxiv} in the context of vector spin models.
We further demonstrate that our lower bound is actually an equality in many cases, including: at high temperatures for all models, at all temperatures for convex models, and at zero temperature for pure bipartite spherical models.
When translated to a statement about the injective norm of a nonsymmetric Gaussian tensor, our lower bound matches an upper bound recently established by Dartois and McKenna \cite{dartois_mckenna24_arxiv}.
\end{abstract}

\maketitle


\tableofcontents


\section{Introduction}

\subsection{Background and motivation} \label{subsec_motivation}
One of the central mathematical achievements in spin glass theory is verification of Parisi's formula for the limiting free energy 
of the \mbox{Sherrington--Kirkpatrick (SK)} model.
This variational formula was identified non-rigorously by Parisi~\cite{parisi79,parisi80a,parisi80b,parisi83} and ultimately proved two decades later in the celebrated works of Guerra~\cite{guerra03} and \mbox{Talagrand~\cite{talagrand06a}}.
The proof strategy was refined en route to generalizations for spherical and mixed $p$-spin models \cite{talagrand06b,panchenko14,chen13}, with key intervening developments including the \mbox{Aizenman--Sims--Starr} scheme~\cite{aizenman_sims_starr03}, Talagrand's positivity principle~\cite{talagrand03c}, and Panchenko's verification of ultrametricity \cite{panchenko13b}.

Recently the Parisi framework has been carried through to yet broader settings: so-called multi-species and vector spin models.
A central difficulty for such models is reinterpreting the Ghirlanda--Guerra identities \cite{ghirlanda_guerra98}, which are an essential component of the modern theory.
In response to this difficulty, Panchenko developed a powerful ``synchronization'' approach that delivered Parisi-type variational formulas for both multi-species and vector spin models \cite{panchenko15,panchenko18a,panchenko18b}.
This development led to further generalizations such as a multi-species Potts spin glass \cite{jagannath_ko_sen18}, spherical models with constrained overlaps \cite{ko20}, and multi-species spherical models \cite{bates_sohn22a,bates_sohn22b}.
Nevertheless, all of these works encounter two fundamental hurdles:

\begin{enumerate}[label=\textup{\arabic*.},ref=\textup{\arabic*}]

\item \label{hurdle_1} The free energy upper bound---using Guerra interpolation---requires a convexity assumption on the covariance structure of the model.
This restricts the class of models that can be rigorously treated, excluding very natural cases of interest such as the bipartite SK model \cite{barra_genovese_guerra11,barra_galluzzi_guerra_pizzoferrato_tantari14}.
Moreover, Guerra-type upper bounds are generally believed to be \textit{incorrect} for nonconvex models, rather than simply unproven.

\item \label{hurdle_2} Meanwhile, the lower bound---using the Aizenman--Sims--Starr scheme---introduces a functional order parameter that is significantly more complicated than in the classical Parisi formula.  Namely, the variational parameter is a probability measure on a certain higher-dimensional space of matrices, rather than on $\R$.
This makes the generalized models more difficult to analyze, and also harder to relate to their classical single-species/one-dimensional counterparts.

\end{enumerate}

Overcoming these hurdles is an active area of research.
Regarding hurdle~\ref{hurdle_1}, there are upper bounds not requiring convexity \cite{mourrat21b,mourrat23} that differ (at least in certain cases) with what Guerra interpolation would predict.
These upper bounds arise from the influential program initiated in \cite{mourrat22}, wherein the Parisi formula is reinterpreted as the solution to a Hamilton--Jacobi equation in Wasserstein space \cite{chen_mourrat25,dominguez_mourrat24,chen24_arxiv,chen25_arxiv}.
It has been conjectured that these upper bounds are sharp.
In the special case of the bipartite spherical SK model, the limiting free energy has been computed explicitly \cite{auffinger_chen14,baik_lee20}, but the Ising version remains an important open problem.
For other pure spherical models, if convergence of free energy is assumed, then its limiting value can be computed from a TAP representation \cite{subag25,subag23c}.

Regarding hurdle~\ref{hurdle_2}, one recent development was realized for the Potts spin glass \cite{elderfield_sherrington83a}, which is the vector analogue of the classical SK model. 
In a ``balanced'' version of this model, it was proved in \cite{bates_sohn24} that the Parisi formula from \cite{panchenko18a} can be reduced to an optimization over probability measures on $\R$ rather than $\R^{d(d-1)/2}$, where $d$ is the ambient dimension of the vector spins.
This result has been adapted to models where a self-overlap correction replaces the balanced constraint: first in the Potts case \cite{chen24c} and then to general permutation-invariant vector spin models \cite{issa24_arxiv}.
These works demonstrated that a spin glass model's symmetry can be exploited to simplify its associated Parisi formula, a theme also present in this paper.
One hope is that the results of \cite{bates_sohn24,chen24c,issa24_arxiv} can be used to simplify the expression for the canonical Potts spin glass (i.e.\ without balancing or self-overlap correction), although it was recently observed in \cite{mourrat25} that the canonical model behaves differently for large enough values of $d$.

In this work, we introduce a class of multi-species spin glasses, which we call \textit{balanced} models (Definition~\ref{def_balanced}), for which we prove a free energy lower bound with an appealing form: the free energy of a particular \textit{single}-species model (Theorem~\ref{thm_main}).
Here we highlight a few features of this result:
\begin{itemize}
     \item It applies to both Ising and spherical models. In the Ising case, the lower bound is similar to work of Issa \cite{issa24_arxiv}; see Remark~\ref{rem_compare}.
    The addition of spherical models leads to several cases where our lower bound is an equality even for non-convex models at low temperature.
    These are mentioned in the next two bullet points.
    \item For convex models, our lower bound is actually the correct limiting free energy (Corollary~\ref{cor_convex}), thereby simplifying formulas from previous works in a fashion similar to \cite{bates_sohn24,chen24_arxiv,issa24_arxiv}; see Remark~\ref{rmk_spar} for further discussion.
    In the non-convex case, our lower bound is correct at least at high temperatures (Corollary~\ref{cor_high}), which cannot be said of replica symmetry breaking bounds obtained from the Aizenman--Sims--Starr scheme.
    Regarding low temperatures, we do not have a general result, but we are not aware of any example where our lower bound is provably \textit{not} an equality. 
    In fact, for the bipartite spherical SK model (Example~\ref{rem_bipartite_SK}), equality holds at all temperatures.
    \item At zero temperature, the result leads to two more new limit theorems in non-convex models: the ground state energy of pure bipartite spherical models in the balanced case (Corollary~\ref{cor_bpbs}), and the injective norm of nonsymmetric Gaussian tensors (Corollary~\ref{cor_inj_norm}).
    Upper bounds for these quantities were established in \cite{mckenna24, dartois_mckenna24_arxiv} by calculating the first moment of the number of critical points using the Kac--Rice formula. 
    This paper offers a simple approach to obtain matching lower bounds, bypassing second moment calculations (see Remark~\ref{rmk:sec:moment}).
\end{itemize}

For these reasons, our result can be viewed as bypassing both hurdles~\ref{hurdle_1} and \ref{hurdle_2} for the special balanced case.
Moreover, the argument is elementary and is nearly self-contained: the only major input is a multi-species version of Talagrand's positivity principle \cite{talagrand03c}.

\subsection{Multi-species spin glasses}
Let $\mathscr{S}$ be a finite set.
Each element of $\mathscr{S}$ is referred to as a species.
Suppose that for each positive integer $N$, there is a partition of the set $\{1,\dots,N\}$ into the various species: 
\eq{
\{1,\dots,N\} = \biguplus_{s\in\mathscr{S}} \II_{s,N}.
}
We denote the fraction of coordinates belonging to species $s$ by $\lambda_{s,N} = \#\II_{s,N}/N$, and assume this fraction converges to a positive number as $N\to\infty$:
\begin{align} \tag{H1} \label{converging_ratios}
 \lim_{N \to \infty} \lambda_{s,N} 
 \eqqcolon\lambda_{s}>0 \quad \text{for each $s\in\mathscr{S}$}.
\end{align}
We write $\vc\lambda_N = (\lambda_{s,N})_{s\in\SSS}$ and $\vc\lambda = (\lambda_s)_{s\in\SSS}$ for the vectors of these ratios. 
We necessarily have $\sum_{s\in\SSS}\lambda_{s,N} = 1 = \sum_{s\in\SSS}\lambda_s$.

The next parameter is a collection of nonnegative numbers $\vc\Delta = (\Delta_{s_1,\dots,s_p})_{s_1,\dots,s_p\in\SSS,\,p\ge1}$ governing the $p$-spin interaction strengths between species.
We assume without loss of generality\footnote{
The coefficient of $\prod_{s\in\SSS}(\lambda_{s,N} x_s)^{p_s}$ in \eqref{cov_b} is the sum of $(\sum_s p_s)!/\prod_s (p_s!)$ many entries of the array $(\Delta^2_{s_1,\dots,s_p})_{s_1,\dots,s_p}$, namely those entries satisfying $\#\big\{j\in\{1,\dots,p\}:\, s_j = s\big\} = p_s$ for each $s$.
By replacing each of these entries with their arithmetic mean, we obtain a \textit{symmetric} model that has the same distribution.}  that $(\Delta_{s_1,\dots,s_p})_{s_1,\dots,s_p\in\SSS}$ is symmetric for each $p$: $\Delta_{s_1,\dots,s_p} = \Delta_{s_{\prt(1)},\dots,s_{\prt(p)}}$ for any permutation $\prt$ on $p$ elements.
We write $\Delta_{s_1,\dots,s_p}^2$ for the square of the number $\Delta_{s_1,\dots,s_p}$.

The final parameter is the choice of configuration space.
We consider either of the two typical cases:
\begin{itemize}
    \item ($\is$) $\Sigma_{N,\is}$ is the discrete hypercube $\{-1,+1\}^N$, and $\mu_{N,\is}$ is uniform measure on this hypercube.
    \item ($\sph$) $\Sigma_{N,\sph}$ is the product of spheres $\otimes_{s\in\mathscr{S}}S_{\# \II_{s,N}}$, where 
    \eq{
    S_{n} \coloneqq \{x\in\R^n:\, \|x\|_2^2 = n\},
    }
    and $\mu_{N,\sph}$ is the corresponding product of uniform measures.
\end{itemize}
Often we will not need to distinguish between these two cases, so we just use a placeholder symbol $\sig\in\{\is,\sph\}$.
We write an element of $\Sigma_{N,\sig}$ as $\sigma = (\sigma_1,\dots,\sigma_N)$, where coordinate $\sigma_i$ is associated to species $s$ if $i\in \II_{s,N}$.
Note that $\Sigma_{N,\is}\subseteq\Sigma_{N,\sph}$, so for either choice of $\sig\in\{\is,\sph\}$, we have
\eeq{ \label{91knbx}
\sum_{i\in \II_{s,N}}\sigma_i^2 = \# \II_{s,N} = \lambda_{s,N}\cdot N \quad \text{for every $s\in\mathscr{S}$, $\sigma\in\Sigma_{N,\sig}$.}
}

We are now ready to define the spin glass model.
The multi-species mixed $p$-spin Hamiltonian is the random function $H_N\colon\Sigma_{N,\sig}\to\R$ given by
\begin{equation}\label{MSKHamiltonian}
H_{N}(\sigma) 
\coloneqq \sum_{p=1}^\infty\frac{1}{N^{(p-1)/2}}\sum_{s_1,\dots,s_p\in\mathscr{S}}\sqrt{\Delta^2_{s_1,\dots,s_p}}\sum_{i_1\in \II_{s_1,N}, \dots, i_p\in \II_{s_p,N}} g_{i_1,\dots,i_p}\sigma_{i_1}\cdots\sigma_{i_p},
\end{equation}
where $(g_{i_1,\dots,i_p})_{1\le i_1,\dots,i_p\le N,\,p\ge1}$ are independent standard normal random variables. 
The Gaussian process $(H_N(\sigma))_{\sigma\in\Sigma_{N,\sig}}$ has mean zero and covariance
\eeq{ \label{cov_a}
\E[H_N(\sigma)H_N(\tau)] = N\xi_N\big(\vc R(\sigma,\tau)\big),
}
where $\xi_N\colon[-1,1]^\SSS\to\R$ is given by
\eeq{ \label{cov_b}
\xi_N(\vc x) \coloneqq \sum_{p=1}^\infty \sum_{s_1,\dots,s_p\in\SSS}\Delta^2_{s_1,\dots,s_p}\lambda_{s_1,N}\cdots\lambda_{s_p,N}x_{s_1}\cdots x_{s_p},
}
and $\vc R(\sigma,\tau) = (R_s(\sigma,\tau))_{s\in\SSS}$ is the vector of intra-species overlaps:
\eeq{ \label{overlap_def}
R_s(\sigma,\tau) \coloneqq \frac{1}{\# \II_{s,N}}\sum_{i\in \II_{s,N}}\sigma_i\tau_i.
}
Note that by Cauchy--Schwarz and \eqref{91knbx}, we always have $\vc R(\sigma,\tau)\in[-1,1]^\SSS$. 
To ensure well-definedness (and analyticity) of $\xi_N$ on this domain, we assume the following decay condition on the entries of $\vc\Delta^2$:
\eeq{ \label{delta_decay_finite}
\sum_{p=1}^\infty(1+\eps)^p\sum_{s_1,\dots,s_p\in\SSS}\Delta^2_{s_1,\dots,s_p}\lambda_{s_1,N}\cdots\lambda_{s_p,N} <\infty \quad \text{for some $\eps>0$.}
}
Since $\lambda_{s,N}\to\lambda_s$ as $N\to\infty$, this condition is implied (for all sufficiently large $N$) by the following statement:
\begin{align} \tag{H2} \label{delta_decay}
\sum_{p=1}^\infty(1+\eps)^p\sum_{s_1,\dots,s_p\in\SSS}\Delta^2_{s_1,\dots,s_p}\lambda_{s_1}\cdots\lambda_{s_p} <\infty \quad \text{for some $\eps>0$.}
\end{align}

We assume \eqref{converging_ratios} and \eqref{delta_decay} throughout the paper; these are the usual hypotheses to make sure the model is well-posed. 
Under these assumptions, we have the following uniform convergence:
\eeq{ \label{uniform_cov_convergence}
\lim_{N\to\infty}\sup_{\vc x\in[-1,1]^\SSS}|\xi_N(\vc x)-\xi(\vc x)| = 0,
}
where the limiting covariance function is
\eeq{ \label{cov_c}
\xi(\vc x)
\coloneqq \sum_{p=1}^\infty \sum_{s_1,\dots,s_p\in\SSS}\Delta^2_{s_1,\dots,s_p}\lambda_{s_1}\cdots\lambda_{s_p}x_{s_1}\cdots x_{s_p}.
}
Define the free energy
\eeq{ \label{FNZN_def}
F_{N,\sig}(\vc\Delta,\vc\lambda_N) 
&\coloneqq\frac{1}{N}\mathbb{E} \log Z_{N,\sig}(\vc\Delta,\vc\lambda_N), \\
\text{where}\quad 
Z_{N,\sig}(\vc\Delta,\vc\lambda_N) &\coloneqq\int_{\Sigma_{N,\sig}}\exp H_{N}(\sigma)\ \mu_{N,\sig}(\dd\sigma).
}
Also define the ground state energy as
\eeq{ \label{GSE_def}
\GSE_{N,\sig}(\vc\Delta,\vc\lambda_N)
\coloneqq\frac{1}{N}\E\Big(\max_{\sigma \in \Sigma_{N,\sig}} H_N(\sigma)\Big).
}
In the classical single-species case $\#\SSS=1$, the array $(\Delta_{s_1,\dots,s_p})_{s_1,\dots,s_p\in\SSS}$ is just a single number (typically denoted $\beta_p$), and the vector $\vc\lambda_N$ is just the single number $1$.
To distinguish this special case, we denote the corresponding free energy and ground state energy by $F^{\single}_{N,\sig}(\bbeta)$ and $\GSE^{\single}_{N,\sig}(\bbeta)$, where $\bbeta = (\beta_1,\beta_2,\dots)$. 

\subsection{Main result}
Our interest is in the values of 
\eq{
\lim_{N\to\infty} F_{N,\sig}(\vc\Delta,\vc\lambda_N) \quad \text{and} \quad \lim_{N\to\infty}\GSE_{N,\sig}(\vc\Delta,\vc\lambda_N),
}
should these limits exist.
For the Ising model without any additional assumptions on $\vc\Delta$, Panchenko \cite{panchenko15} proved a lower bound in terms of a multi-species Parisi variational formula.
A similar result for the spherical model was later proved by the authors \cite{bates_sohn22a}.
Our main result (Theorem~\ref{thm_main}) provides a new alternative lower bound in the following special case.
\begin{definition}
\label{def_balanced}
We say the multi-species parameter pair $(\vc\Delta,\vc\lambda)$ is \textit{balanced} if
\begin{equation}  \tag{H3} \label{balanced}
\begin{split}
\Delta^2_{t} \quad &\text{does not depend on $t\in\SSS$ for $p=1$, and} \\
\sum_{s_2,\dots,s_p\in\SSS}\Delta^2_{t,s_2,\dots,s_p}\lambda_{s_2}\cdots\lambda_{s_p} \quad &\text{does not depend on $t\in\SSS$, for every $p\ge2$.}
\end{split}
\end{equation}
\end{definition}

\begin{example} \label{ex_model}
Definition~\ref{def_balanced} includes the following examples:
\begin{enumerate}[label=\textup{(\alph*)}]

\item \label{ex_model_1} 
Any model in which the species are ``exchangeable'': $\lambda_s = \frac{1}{\#\SSS}$ for every $s\in\SSS$, and $\Delta_{s_1,\dots,s_p} = \Delta_{\pi(s_1),\dots,\pi(s_p)}$ for every permutation $\pi$ on $\SSS$.  
For instance, the bipartite SK model \eqref{bipartite_parameters} fits this description.
The $p$-partite model in Lemma~\ref{lem_translate} also fits this description.

\item \label{ex_model_2}
Any model such that for every $p$, the value of $\Delta_{s_1,\dots,s_p}$ depends only on $p$ (say it is equal to $\beta_p$).
In the Ising case, this condition reduces the model to a single species.
But in the spherical case, there is still a distinction: $\Sigma_{N,\sph}$ is a product of spheres rather than a single sphere.
We are not aware of any special treatment of this scenario (i.e.\ constant interaction strengths, but on a product of spheres) in the literature, but the associated covariance function 
\eq{
\xi(\vc x) = \sum_{p=1}^\infty \beta_p^2\Big(\sum_{s\in\SSS}\lambda_sx_s\Big)^p
}
is convex on $[0,1]^\SSS$ and thus included in Corollary~\ref{cor_convex} below.

\item \label{ex_model_3}
Any pure model (i.e.\ $\xi(\vc x) \propto \prod_{s\in\SSS}x_s^{p_s}$ for some exponents $p_s\ge1$) such that
\eeq{ \label{3r38x}
\frac{1}{\lambda_t}\frac{\partial\xi}{\partial x_t}(\vc 1) \quad \text{does not depend on $t\in\SSS$},
}
where $\vc 1 = (1,\dots,1)$.
For instance, the models in Example~\ref{rem_bipartite_pure} satisfy this condition.
In general, \eqref{balanced} implies \eqref{3r38x} because
\eq{
\frac{1}{\lambda_t}\frac{\partial\xi}{\partial x_t}(\vc x) = \Delta^2_t + \sum_{p=2}^\infty \sum_{s_2,\dots,s_p\in\SSS}p\Delta^2_{t,s_2,\dots,s_p}\lambda_{s_2}\cdots\lambda_{s_p}x_{s_2}\cdots x_{s_p}.
}
For pure models the reverse implication is also true, since there is only one value of $p$ (namely $p = \sum_{s\in\SSS}p_s$) such that $\Delta_{t,s_2,\dots,s_p}$ can take nonzero values.
More generally, the implication \eqref{3r38x}$\implies$\eqref{balanced} holds whenever $\xi$ is a homogeneous polynomial of degree $p$, in which case
the parameter $\beta_p^2$ defined in \eqref{beta_def} is simply $\xi(\vc 1)$.
\QED

\end{enumerate}
\end{example}

The following is our main result and will be proved in Section~\ref{sec_lower_bound}.
As mentioned before, hypotheses \eqref{converging_ratios} and \eqref{delta_decay} are standard and ensure that the model is well-posed.
The balanced assumption \eqref{balanced} is the special extra ingredient.

\begin{theorem}[Lower bound for balanced models]
\label{thm_main}
Assume \eqref{converging_ratios}, \eqref{delta_decay}, \eqref{balanced}. 
Then for either $\sig\in\{\is,\sph\}$ we have
\eeq{\label{eq:thm:msk}
\liminf_{N \to \infty} F_{N,\sig}(\vc\Delta,\vc\lambda_N)\geq \lim_{N\to\infty} F_{N,\sig}^\single(\bbeta),
}
where the coordinates of $\bbeta = (\beta_p)_{p\ge1}$ are given by
\eeq{ \label{beta_def}
\beta_p^2 \coloneqq \sum_{s_1,\dots,s_p\in\SSS}
\Delta^2_{s_1,\dots,s_p}\lambda_{s_1}\cdots\lambda_{s_p}.
}
Furthermore,
\eeq{ \label{GSE_lower}
\liminf_{N\to\infty} \GSE_{N,\sig}(\vc\Delta,\vc\lambda_N)\geq \lim_{N\to\infty} \GSE_{N,\sig}^{\single} (\bbeta).
}
\end{theorem}

For the rest of the paper, we let $\beta_p= \beta_p(\vc\Delta,\vc\lambda)$ be defined as in \eqref{beta_def}.
The collection $\bbeta = (\beta_p)_{p\ge1}$ is used in Theorem~\ref{thm_main} as the inverse temperature parameters for a \textit{single}-species model to which the balanced multi-species model is compared.
When $\xi$ is convex, our lower bound is in fact an equality (Corollary~\ref{cor_convex}) as were the lower bounds in \cite{panchenko15,bates_sohn22a} obtained from the Aizenman--Sims--Starr scheme, but our expression is much simpler because of the balanced hypothesis (see Remark~\ref{rmk_spar}).
When $\xi$ is not convex, our lower bound can in fact be larger than those from \cite{panchenko15,bates_sohn22a}.
For example, in the bipartite Ising SK model (same as Example~\ref{rem_bipartite_SK} but on the hypercube instead of a product of spheres), our lower bound is correct at high temperature (Corollary~\ref{cor_high}), but the Aizenman--Sims--Starr scheme is not: the limiting value of $\xi(\vc 1)/2$ is obtained at a saddle point of the replica symmetric expression, not a minimizer \cite[Remark~5]{barra_genovese_guerra11}.
See \cite[Section~1.5]{bates_sloman_sohn19_arxiv} for further discussion.

\begin{remark}[Earlier results by Issa]\label{rem_compare}
For Ising models with ``exchangeable'' species (Example~\ref{ex_model}\ref{ex_model_1}), the lower bound \eqref{eq:thm:msk} is already known from \cite[Theorem~7.6]{issa24_arxiv}.
This is because such models can be treated as permutation-invariant vector spin glasses.
Furthermore, \cite[Theorem~7.10]{issa24_arxiv} shows that this lower bound is an equality if the viscosity solution to a certain equation is Gateaux differentiable.
This differentiability is known when $\xi$ is convex \cite{chen_mourrat25}, so \cite[Theorem~1.3]{issa24_arxiv} can used to obtain our Corollary~\ref{cor_convex} in the Ising case with exchangeable species, although our route to this corollary is different.
Our proof of Theorem~\ref{thm_main}, however, uses a strategy similar to \cite[Section~7]{issa24_arxiv}, but we allow for spherical models and for non-equal species ratios.
\QED
\end{remark}

For the proof of Theorem~\ref{thm_main}, we use a (completely general) one-sided bound \eqref{rh6ef1z}, obtained from Guerra-style interpolation together with Talagrand's positivity principle. 
The key observation is that the balanced assumption allows us to know the sign of the error term in this bound, via the simple but crucial calculation in Proposition~\ref{key_claim}.
Moreover, the equality statement in that claim suggests that our choice of $\bbeta$ is optimal.
Indeed, we provide in Section~\ref{sec_apps} several cases in which there is an upper bound that matches the lower bound from Theorem~\ref{thm_main}, which leads to new limit theorems. See also Remark~\ref{rmk:DM25} for an independent concurrent work.

\begin{remark}[Convergence of free energies] \label{rmk_conv}
Without further assumptions on $\xi$ such as convexity (see Proposition~\ref{prop:guerra}), it has not been proved that the multi-species quantities $F_{N,\sig}(\vc\Delta,\vc\lambda_N)$ and $\GSE_{N,\sig}(\vc\Delta,\vc\lambda_N)$ converge as $N\to\infty$.
On the other hand, convexity on $[0,1]^\SSS$ is automatic when $\#\SSS=1$, so the single-species quantities $F_{N,\sig}^\single(\bbeta)$ and $\GSE^\single_{N,\sig}(\bbeta)$ are known to converge (provided $\bbeta$ satisfies \eqref{beta_decay} for some $\eps>0$, which follows from \eqref{delta_decay}). See Remark~\ref{rmk_spar} for relevant references.
\QED
\end{remark}

\begin{remark}[Concentration] \label{rmk_conc}
By standard Gaussian concentration techniques, it is well known that both the free energy and the ground state energy concentrate near their mean.
More specifically, \cite[Lemma~3]{panchenko07} gives 
\eq{
\P\Big(\Big|\frac{\log Z_{N,\sig}(\vc\Delta,\vc\lambda_N)}{N} - F_{N,\sig}(\vc\Delta,\vc\lambda_N)\Big| \ge t\Big) 
\le 2\exp\Big(\frac{-Nt^2}{4\xi_N(\vc 1)}\Big) 
\quad \text{for all $t\ge0$.}
}
Meanwhile, the Borell--TIS inequality \cite[Theorem~5.8]{boucheron_lugosi_massart13} gives
\eeq{ \label{borell_tis}
\P\Big(\Big|\max_{\sigma \in \Sigma_{N,\sig}}\frac{H_N(\sigma)}{N} - \GSE_{N,\sig}(\vc\Delta,\vc\lambda_N)\Big|\ge t\Big) 
\le 2\exp\Big(\frac{-Nt^2}{2\xi_N(\vc 1)}\Big)
\quad \text{for all $t\ge0$.}
}
These concentration results imply that \eqref{eq:thm:msk} and \eqref{GSE_lower} remain true almost surely if we remove the expectation $\E$ from \eqref{FNZN_def} and \eqref{GSE_def}.
\QED
\end{remark}

\begin{remark}[Regularity assumption]
Whenever considering a spin glass Hamiltonian with infinitely many terms, one needs a decay condition on the coefficients in order for the model to be well-defined.
Our condition \eqref{delta_decay} is meant to be essentially optimal, although in the literature often $1+\eps$ is replaced by $2$ simply for convenience.
Any such instance does not genuinely affect the theory, so we will quote results from other papers without further mention of this technical detail.
For an example of how \eqref{delta_decay} is sufficient, see the arXiv version of \cite{bates_sohn22a_arxiv}.
\QED
\end{remark}

\begin{remark}[External fields]
Theorem~\ref{thm_main} also holds when one includes an external field of the form $h\sum_{i=1}^{N}\sigma_i$ in both the multi-species and single-species Hamiltonians, where $h\in\R$ is a constant.
The only modification is notational: in the proof of Proposition~\ref{prop_key}, the first line of \eqref{8g2bfd3} would consist of covariances rather than expectations of products, but the second line would remain the same.
We set $h=0$ for simplicity.
\QED
\end{remark}

\begin{remark}[Mourrat's upper bound]
It would be interesting to see if the upper bound from \cite[Theorem~1.1]{mourrat23} matches our lower bound \eqref{eq:thm:msk} in some nonconvex case.
As mentioned in Remark~\ref{rem_compare}, \cite[Theorem~1.5]{issa24_arxiv} proves equality for exchangeable models under an unverified but plausible hypothesis; see \cite[Remark~7.12]{issa24_arxiv}.
Of particular interest is the bipartite Ising SK model (same as Example~\ref{rem_bipartite_SK} but on the hypercube instead of a product of spheres), which is a major open problem.
\QED
\end{remark}

\section{Applications} \label{sec_apps}

This section gives several scenarios in which the lower bound from Theorem~\ref{thm_main} agrees with an upper bound already in the literature.
With the exception of the bipartite SK model (Example~\ref{rem_bipartite_SK}), the resulting limit theorems are new.

\subsection{High-temperature regime}
The first and easiest scenario in which our lower bound is exact is when the inverse-temperature parameters $\bbeta = (\beta_p)_{p\ge1}$ are sufficiently small that the limiting quenched free energy agrees with the annealed free energy.
More precisely, by Jensen's inequality we always have
\eeq{ \label{er8bx}
F_{N,\sig}(\vc\Delta,\vc\lambda_N) 
&\stackrefp{FNZN_def}{<} \frac{1}{N}\log \E Z_{N,\sig}(\vc\Delta,\vc\lambda_N) \\
&\stackref{FNZN_def}{=} \frac{1}{N}\log\int_{\Sigma_{N,\sig}}\E\exp H_N(\sigma)\ \mu_{N,\sig}(\dd\sigma)
\stackref{cov_a}{=} \frac{\xi_N(\vc 1)}{2}.
}
When strictness is maintained in the large-$N$ limit, we say the model is in the \textit{low-temperature phase}.
If instead the limits of the two sides agree, we say the model is in the \textit{high-temperature phase}.
The following result says that a balanced multi-species model is in the high-temperature phase whenever the corresponding single-species model is in its high-temperature phase.

\begin{corollary}[Equality at high temperature] \label{cor_high}
Assume \eqref{converging_ratios}, \eqref{delta_decay}, \eqref{balanced}.
With $\beta_p$ as in \eqref{beta_def}, define
\eeq{ \label{cov_single}
\xi^\single(x) \coloneqq \sum_{p=1}^\infty \beta_p^2x^p, \quad x\in[-1,1].
}
For either $\sig\in\{\is,\sph\}$, the following implication holds:
\eeq{ \label{2gmmb}
\lim_{N\to\infty}F_{N,\sig}^\single(\bbeta) = \frac{\xi^\single(1)}{2} \quad \implies \quad
\lim_{N\to\infty} F_{N,\sig}(\vc\Delta,\vc\lambda_N) = \frac{\xi^\single(1)}{2}. 
}
\end{corollary}

\begin{proof}
When the hypothesis of \eqref{2gmmb} holds, we have
\eq{
\frac{\xi^\single(1)}{2} = \lim_{N\to\infty}F_{N,\sig}^\single(\bbeta)
&\stackref{eq:thm:msk}{\le} \liminf_{N\to\infty} F_{N,\sig}(\vc\Delta,\vc\lambda_N) \\
&\stackrefp{eq:thm:msk}{\le}\limsup_{N\to\infty} F_{N,\sig}(\vc\Delta,\vc\lambda_N) \\
&\stackrefpp{er8bx}{eq:thm:msk}{\le}\limsup_{N\to\infty}\frac{\xi_N(\vc1)}{2}
\stackref{uniform_cov_convergence}{=} \frac{\xi(\vc1)}{2}.
}
To complete the proof, observe that $\xi(\vc1)=\xi^\single(1)$ by comparing \eqref{cov_c} and \eqref{cov_single}.
\end{proof}

One naturally wonders if the converse of \eqref{2gmmb} is also true.
For the single-species Ising SK model, the critical value separating the high- and low-temperature phases is $\beta_\cc = \frac{1}{\sqrt{2}}$.\footnote{The lower bound $\beta_\cc\ge\frac{1}{\sqrt{2}}$ is from \cite[Proposition~2.1(ii)]{aizenman_lebowitz_ruelle87}, 
and the upper bound $\beta_\cc\le\frac{1}{\sqrt{2}}$ is from \cite{toninelli02}.
Regarding the latter reference, one can also see \cite[Theorem~13.3.1]{talagrand11b} with external field $h=0$ and $q=0$. 
The Hamiltonian in \cite{aizenman_lebowitz_ruelle87,toninelli02,talagrand11b} only sums over indices $i<j$, so their model at inverse temperature $\beta$ is equivalent to our two-spin single-species model at inverse temperature $\frac{\beta}{\sqrt{2}}$.
Their free energy also has an additional summand of $\log 2$ that disappears here because our partition function $Z_{N,\is}$ in \eqref{FNZN_def} includes a factor of $2^{-N}$ from the presence of the uniform probability measure $\mu_{N,\is}$ on $\{-1,+1\}^N$.}
Under the definition \eqref{beta_def}, this critical value converts to
\eeq{ \label{re8gc}
\sum_{s,t\in\SSS}2\Delta^2_{s,t}\lambda_{s}\lambda_{t} = 1.
}
When the model is balanced, \eqref{re8gc} coincides with the multi-species critical condition found in \cite{bates_sloman_sohn19} and generalized in \cite[display~(14)]{dey_wu21}.\footnote{In \cite{dey_wu21}, the Hamiltonian is again restricted to indices $i<j$.  This explains the extra factor of 2 in \eqref{re8gc}.  Also, in the notation of \cite{dey_wu21}, our balanced condition says that all the row sums of $\Delta^2\Lambda$ are identical, so the Perron--Frobenius theorem implies that the spectral radius of this matrix is equal to this row sum.}
If that critical value is correct (see the discussion in \cite[Section~8]{dey_wu21}), then the converse of \eqref{cov_single} is true for the multi-species SK model.

For other more general mixed $p$-spin models, the separation between high- and low-temperature phases is less understood.
Nevertheless, Corollary~\ref{cor_high} provides a mechanism for immediately using high-temperature conditions identified for single-species models, such as \cite[Theorem~16.3.1]{talagrand11b} in the Ising case, and \cite[Proposition~2.3]{talagrand06b} in the spherical case.

\subsection{Convex case} 
When the covariance function $\xi$ is convex, we can go beyond the high-temperature regime to all temperatures.
In fact, we only need convexity on the positive orthant, thanks to a multi-species version of Talagrand's positivity principle stated in Lemma~\ref{pos_principle}.
The following result is proved in Section~\ref{subsec_upper_bound}.

\begin{corollary}[Equality in the convex case]\label{cor_convex}
Assume \eqref{converging_ratios}, \eqref{delta_decay}, \eqref{balanced}.
If the covariance function $\xi$ from \eqref{cov_c} is convex on $[0,1]^\SSS$, then for either $\sig\in\{\is,\sph\}$ we have
\eeq{ \label{eq_msk2_a}
\lim_{N\to\infty} F_{N,\sig}(\vc\Delta,\vc\lambda_N)=\lim_{N\to\infty} F_{N,\sig}^\single(\bbeta),
}
where $\bbeta = (\beta_p)_{p\ge1}$ is given by \eqref{beta_def}.
Furthermore,
\eeq{ \label{eq_msk2_b}
\lim_{N\to\infty} \GSE_{N,\sig}(\vc\Delta,\vc\lambda_N)
=\lim_{N\to\infty} \GSE_{N,\sig}^\single(\bbeta).
}
\end{corollary}

\begin{remark}[Single-species formulas] \label{rmk_spar}
The right-hand sides of \eqref{eq_msk2_a} and \eqref{eq_msk2_b} are given by Parisi variational formulas: 
\begin{itemize}
    \item In the Ising case, see \cite{talagrand06a,panchenko14} for the free energy, and \cite{auffinger_chen17} for the ground state energy.
    \item In the spherical case, see \cite{talagrand06b,chen13,huang_sellke23b_arxiv} for the free energy, and \cite{chen_sen17,jagannath_tobasco17b} for the ground state energy.
\end{itemize}
Using Corollary~\ref{cor_convex} in tandem with these works yields variational expressions for the left-hand sides of \eqref{eq_msk2_a} and \eqref{eq_msk2_b} that are much simpler than the general formulas established in \cite{panchenko15,bates_sohn22a}.
Nevertheless, our proof will use the upper bounds from \cite{barra_contucci_mingione_tantari15,panchenko15,bates_sohn22a}, recalled as Proposition~\ref{prop:guerra}.
Moreover, an optimizer can be found by simply lifting the optimizer for the single-species formula (see Lemma~\ref{lem:parisi:ftl:reduce}).
Whether the former optimizer enjoys the same uniqueness property as the latter \cite{bovier_klimovsky09,auffinger_chen15b,jagannath_tobasco16} is not clear, but see \cite{chen_issa25_arxiv} for related progress.
\QED
\end{remark}

\begin{remark}[Connections to algorithmic symmetry breaking]
For the multi-species spherical spin glass model, Huang and Sellke characterized the algorithmic threshold energy, demonstrating that approximate message passing achieves this threshold~\cite{huang_sellke24}, whereas no algorithm that is Lipschitz with respect to the disorder can surpass it~\cite{huang_sellke23a_arxiv}. This threshold is expressed through an extended variational principle, initially established for Ising mixed $p$-spin glasses by El Alaoui, Montanari, and Sellke~\cite{elalaoui_montanari_sellke21}. 
In the multi-species setting, the synchronization map from \cite{panchenko15,bates_sohn22a} plays the role of the functional order parameter in this variational principle. 
Interestingly, as observed in \cite[Section~1.4]{huang_sellke23a_arxiv}, the algorithmic threshold for some balanced spherical spin glasses (even some which are convex and exchangeable in the sense of Example~\ref{ex_model}\ref{ex_model_1}) can be attained by a synchronization map which is not symmetric in the species, a phenomenon termed \emph{algorithmic symmetry breaking}. 
While Corollary~\ref{cor_convex} establishes the equivalence of the balanced multi-species model to a single-species model from the free energy perspective, this algorithmic symmetry breaking hints at potentially significant qualitative differences in their Gibbs measures. 
We leave this intriguing possibility for future investigation.
\QED
\end{remark}

\subsection{Bipartite spherical models} 
The only nonconvex model for which $(F_{N,\sig}(\vc\Delta,\vc \lambda_N))_{N\ge1}$ is known to converge at all temperatures (i.e.\ for any scalar multiple of $\vc\Delta$) is the bipartite spherical model \cite{baik_lee20}.
In the balanced version of this model, our lower bound is actually the correct free energy.
We document this fact in the following example. 

\begin{example}[Balanced bipartite spherical SK model] \label{rem_bipartite_SK}
Consider the two-species case $\SSS = \{1,2\}$ with equal-size species and only cross-species two-spin interaction:
\eeq{ \label{bipartite_parameters}
\lambda_1=\lambda_2=\tfrac12, \qquad
\xi(x_1,x_2) = \beta^2 x_1x_2.
}
Regarding $\beta^2$ as $\beta_2^2$ from \eqref{beta_def}, this corresponds to
\eq{
\Delta_{1,2}^2 = 2\beta^2, \qquad
\Delta_{1,1} = \Delta_{2,2} = 0, \qquad
\Delta_{s_1,\dots,s_p} = 0 \quad \text{for all $p\neq2$.}
}
It is known that if $\lambda_{1,N} \to \tfrac12$ as $N\to\infty$, then
\eeq{ \label{uc7n}
\lim_{N\to\infty}F_{N,\sph}(\vc\Delta,\vc\lambda_N) 
= \lim_{N\to\infty} F_{N,\sph}^\single(\beta). 
}
So in this case, the inequality \eqref{eq:thm:msk} is in fact an equality.
More precisely, both sides of \eqref{uc7n} are equal to
\eeq{ \label{ehc6b}
\begin{cases}
\tfrac12\beta^2 &\text{if $0\le\beta\le \frac{1}{\sqrt{2}}$} \\
\sqrt{2}\beta - \tfrac12\log(\sqrt{2}\beta)-\tfrac34 &\text{if $\beta\ge \frac{1}{\sqrt{2}}$.}
\end{cases}
}
The computation of the right-hand side of \eqref{uc7n} was first done non-rigorously in \cite{kosterlitz_thouless_jones76}.
A formula for $p$-spin models (consistent with \cite{kosterlitz_thouless_jones76} when $p=2$) was proposed in \cite{crisanti_sommers92} and then verified rigorously
in \cite{talagrand06b} (allowing only even $p$) and \cite{chen13} (allowing all $p$).
Meanwhile, the left-hand side of \eqref{uc7n} was computed in \cite[Theorem~2.1]{baik_lee20}.\footnote{When multiplied by $\beta$, the Hamiltonian in \cite[display~(1.3)]{baik_lee20} corresponds in our setting to $\xi(x_1,x_2) = \big(\frac{\beta}{2}\big)^2x_1x_2$. For this reason, we replace the $\beta$ appearing in \cite[Theorem~2.1]{baik_lee20} with $2\beta$.}
While \cite{baik_lee20} requires $\lambda_{1,N} = \tfrac12 + O(N^{-1})$ because their law of large numbers follows from a fluctuation result (see \cite[Remark~2.4]{baik_lee20}), there is an alternative formula \cite[Theorem~2.2]{genovese_tantari20} that assumes only $\lambda_{1,N}\to\frac12$.
Simply knowing the convergence of the free energy by \cite{genovese_tantari20} is enough to extend \cite[Theorem~2.1]{baik_lee20} to any sequence $(\lambda_{1,N})_{N\ge1}$ converging to $\frac12$.
\QED
\end{example}

The next example generalizes the previous one.
Although a matching upper bound for the free energy is not presently available in this more general setting, there is a matching upper bound for the ground state energy.
This results in Corollary~\ref{cor_bpbs}, which is stated after introducing the following model.

\begin{example}[Balanced pure bipartite spherical model] \label{rem_bipartite_pure}
Fix integers $p,q\ge 1$.
Consider the two-species case $\SSS = \{1,2\}$ with
\begin{subequations}
\label{wfe78}
\eeq{ \label{wfe78_a}
\lambda_1=\frac{p}{p+q}, \quad \lambda_2 = \frac{q}{p+q}, \qquad
\xi(x_1,x_2) = \beta^2x_1^px_2^q.
}
Regarding $\beta^2$ as $\beta_{p+q}^2$ from \eqref{beta_def}, this corresponds to the model such that
\eeq{ \label{wfe78_b}
\Delta^2_{s_1,\dots,s_{r}} &\neq 0 \quad \text{only when $r=p+q$, $\#\{j:\, s_j = 1\} = p$, $\#\{j:\, s_j = 2\} = q$,} \\
\beta^2 &= \binom{p+q}{p}\frac{p^pq^q}{(p+q)^{p+q}}\Delta^2_{1,\dots,1,2,\dots,2}.
}
\end{subequations}
This is a special case of the model considered in \cite{auffinger_chen14,genovese22}, which computed the limiting free energy for sufficiently small $\beta$.
Here we allow arbitrary $\beta$ and get a one-sided bound from Theorem~\ref{thm_main}.
More specifically, using \cite[Theorem~1.1]{talagrand06b} for the single-species free energy (which is justified by \cite{chen13} when $p+q$ is odd) and invoking \cite[Proposition~2.2]{talagrand06b}, our bound \eqref{eq:thm:msk} becomes
\eeq{ \label{8fvcg}
&\liminf_{N\to\infty} F_{N,\sph}(\vc\Delta,\vc\lambda_N) 
\\
&\ge \inf_{m\in(0,1], a\in[0,1)}\frac{1}{2}\Big[\beta^2\big(1-(1-m)a^{p+q}\big)+\frac{1}{m}\log\Big(1+\frac{ma}{1-a}\Big) + \log(1-a)\Big].
}
When $p=q=1$, the right-hand side of \eqref{8fvcg} can be solved to yield \eqref{ehc6b}.
For $p+q\ge3$, the right-hand side of \eqref{8fvcg} has an alternative expression in \cite[Theorem~2]{subag17b}.
\QED
\end{example}

The ground state energy of the model from Example~\ref{rem_bipartite_pure} is now also accessible.
An upper bound was already established in \cite{mckenna24}, and our lower bound matches it:

\begin{corollary}[Ground state energy for balanced pure bipartite spherical models] \label{cor_bpbs}
Fix integers $p,q\ge2$ and consider the two-species model from \eqref{wfe78}.
Assume the prelimiting volume parameters satisfy
\eeq{ \label{volume_assumption}
N = n(p+q)+2, \quad \#\II_{1,N} = np+1, \quad \#\II_{2,N} = nq+1 \quad
\text{for some $n\in\Z_{\ge0}$.}
}
We then have
\eeq{ \label{cor_bpbs_eq}
\lim_{n\to\infty}\GSE_{N,\sph}(\vc\Delta,\vc\lambda_N) = \GSE_{N,\sph}^\single(\beta).
}
\end{corollary}

The proof of Corollary~\ref{cor_bpbs} is provided in Section~\ref{subsec_proof_bipartite}.

\begin{remark}[Convergence rate of species ratios]
The assumption \eqref{volume_assumption} comes from \cite[display~(4)]{mckenna24}, but it is natural to expect that 
$\lambda_{1,N}\to\frac{p}{p+q}$ 
is sufficient.
Indeed, this is true if one can show that
\eeq{ \label{xkg3x}
\limsup_{N\to\infty}|\GSE_{N,\sph}(\vc\Delta,\vc\lambda_N)-\GSE_{N,\sph} (\vc\Delta,\wt{\vc\lambda}_N)| = 0
}
for any two sequences $(\vc\lambda_N)_{N\ge1}$ and $(\wt{\vc\lambda}_N)_{N\ge1}$ with the same limit \eqref{converging_ratios}.
We do not doubt the validity of \eqref{xkg3x} or the corresponding statement for free energy, but presently we do not have a proof.
\QED
\end{remark}

\begin{remark}[Complexity perspective] \label{rmk_complexity}
The upper bound ($\le$) for \eqref{cor_bpbs_eq} was obtained in \cite{mckenna24} by computing a variational formula for the limiting expected number of critical points of $H_N$.
In the balanced case (i.e.\ $\lambda_1 = p/(p+q)$), this formula could be solved, and then a first-moment method produced the upper bound
\eeq{ \label{93gx}
\lim_{N\to\infty}\GSE_{N,\sig}(\vc\Delta,\vc\lambda_N) \le \beta E_0(p+q),
}
where $E_0$ is the value at which ``complexity vanishes'' in the single-species model; see \cite[display~(2.22)]{auffinger_benarous_cerny13}.
From \cite[Theorem~2.12]{auffinger_benarous_cerny13} (see also \cite[Appendix~D]{subag17a}), it is known that $E_0(p)$ is the limiting ground state energy of a single-species pure $p$-spin model:
\eeq{ \label{3r8gx}
\text{for all $p\geq3$,} \quad 
E_0(p) = \lim_{N\to\infty}\frac{1}{N}\max_{\sigma\in S_N}\frac{1}{N^{(p-1)/2}}\sum_{i_1,\dots,i_p=1}^N g_{i_1,\dots,i_p}\sigma_{i_1}\cdots\sigma_{i_p} \quad \mathrm{a.s.}
}
See \cite[Theorem~1]{subag23a} for an explicit expression for $E_0(p)$, or \cite[Proposition~3]{chen_sen17} for an alternative computation of the right-hand side of \eqref{3r8gx}.
\QED
\end{remark}

\begin{remark}[Hypothesis on number of spins]
\label{rmk:sec:moment}
The assumption $p,q\ge2$ in Corollary~\ref{cor_bpbs} comes from \cite{mckenna24}, although it was suggested in  \cite[Remark~2.11]{mckenna24} that this restriction (i.e.\ avoiding $p=1$ or $q=1$) is technical rather than fundamental.
Incidentally, \eqref{cor_bpbs_eq} remains true for $(p,q)=(1,1)$ because of \eqref{uc7n} and Lemma~\ref{lem_free_to_GSE}.
So the only excluded cases are $(p,1)$ or $(1,q)$.
When $\min(p,q)\ge96$, Kivimae \cite{kivimae23} carried out a second-moment complexity analysis proving that \eqref{93gx} is actually an equality, so Corollary~\ref{cor_bpbs} is not new in this parameter regime.
Moreover, \cite[Corollary~1]{kivimae23} gives a limiting ground state energy for any $\lambda_1 \in (0,1)$, not just $\lambda_1 = p/(p+q)$.
\QED
\end{remark}

\subsection{Injective norm of Gaussian random tensors}
The injective norm of a tensor is a generalization of the spectral norm of a matrix.
There has been recent interest in understanding the injective norm of \textit{random} tensors, e.g.\ \cite{zhou_zhu21,bandeira_gopi_jiang_lucca_rothvoss24_arxiv,boedihardjo24_arxiv}.
Here we give a new asymptotic result for Gaussian random tensors.

For positive integers $p$ (the ``order'') and $d$ (the ``dimension''), consider a tensor $T = (T_{i_1,\dots,i_p})_{1\le i_1,\dots,i_p\le d}$ whose entries are independent standard normal random variables.
Denote the unit sphere in $\R^d$ by 
\eq{
\S^{d-1} \coloneqq \{u\in\R^d:\, \|u\|_2 = 1\}.
}
The \textit{injective norm} of $T$ is defined as
\eeq{ \label{def_inj_norm}
\|T\|_{\inj} 
&\coloneqq \max_{u^{(1)},\dots,u^{(p)}\in\S^{d-1}}\iprod{T}{u^{(1)}\otimes\cdots\otimes u^{(p)}} \\
&= \max_{u^{(1)},\dots,u^{(p)}\in\S^{d-1}}\sum_{i_1,\dots,i_p=1}^d T_{i_1,\dots,i_p}u^{(1)}_{i_1}\cdots u^{(p)}_{i_p}.
 }
As $d\to\infty$, it was determined in \cite{nguyen_drineas_tran15,tomioka_suzuki14} that $\|T\|_\inj = O(\sqrt{d})$.
Recently Dartois and McKenna \cite{dartois_mckenna24_arxiv} gave a limiting upper bound on $\frac{1}{\sqrt{d}}\|T\|_\inj$ and predicted their bound was tight.
Using the lower bound from Theorem~\ref{thm_main}, we now confirm their prediction:

\begin{corollary}[Injective norm of large Gaussian tensor of order $p$] \label{cor_inj_norm}
Assume $p\ge3$ and the entries of $T = (T_{i_1,\dots,i_p})_{1\le i_1,\dots,i_p\le d}$ are independent standard normal random variables.
Then
\eeq{ \label{8unz}
\lim_{d\to\infty} \frac{1}{\sqrt{d}}\|T\|_\inj = \sqrt{p}E_0(p) \quad \mathrm{a.s.},
}
where $E_0(p)$ is the limiting ground state energy from Remark~\ref{rmk_complexity}.
\end{corollary}

The proof is provided in Section~\ref{subsec_proof_inj_norm}.
The reason Theorem~\ref{thm_main} applies is that the injective norm can be regarded as the (quenched) ground state energy of a certain multi-species spin glass; this is recorded in Lemma~\ref{lem_translate}.
The \textit{multi}-species structure is essential because $T$ is \textit{non}symmetric.
If $T$ were symmetrized, then $\|T\|_\inj$ would correspond to a single-species spin glass, and \eqref{8unz} would reduce to the previously known \eqref{3r8gx}.

\begin{remark}[Related work]
\label{rmk:DM25}
In a forthcoming paper by Dartois and McKenna~\cite{dartois_mckenna25_prep}, a similar statement to Corollary~\ref{cor_inj_norm} will be obtained by an entirely different method.
As communicated to us, their result is conditional on a certain variational fact (which is verifiable for small $p$).
On the other hand, their techniques may be able to handle the more complicated case where the entries of $T$ are complex Gaussians. 
Moreover, their methodology will lend insights into the critical points of the function $(u^{(1)}\otimes\cdots\otimes u^{(p)})\mapsto\iprod{T}{u^{(1)}\otimes\cdots\otimes u^{(p)}}$, whereas our approach is purely about the free energy.
\QED
\end{remark}

\begin{remark}[Other regimes of interest]
If $p = 1$, then $\|T\|_\inj$ is simply the $\ell^2$ norm of the vector $T = (T_1,\dots,T_d)$, which grows like $\sqrt{d}$ by the strong law of large numbers.

If $p = 2$, then $\|T\|_\inj$ is the spectral norm of a real $d\times d$ Ginibre matrix with i.i.d.\ $\mathcal{N}(0,1)$ entries, which grows like $2\sqrt{d}$ by \cite[Theorem~3.1]{yin_bai_krishnaiah88}.
Therefore, \eqref{3r8gx} and \eqref{8unz} still hold if we take the convention $E_0(2) = \sqrt{2}$ as in \cite{dartois_mckenna24_arxiv}, although this is a special case differing from the complexity-based definitions discussed in Remark~\ref{rmk_complexity}.
See \cite[Remark~2.11]{auffinger_benarous_cerny13} and \cite[Lemma~2.7]{dartois_mckenna24_arxiv} for further discussion.

Also considered in \cite{dartois_mckenna24_arxiv} is the case of fixed $d$ and $p\to\infty$, which is natural in quantum information, especially when the tensor is complex-valued.
Our result does not address this scenario or allow complex tensors, but \cite[Theorem~1.1]{dartois_mckenna24_arxiv} provides upper bounds in those settings as well.
\QED
\end{remark}

\section{Proof of main result: energy lower bound} \label{sec_lower_bound}

Our proof of Theorem~\ref{thm_main} will compare models with different values of the interaction parameters $\vc\Delta$.
In order to make the dependence on $\vc\Delta$ explicit, we will write $H_{N,\vc\Delta}$, $\xi_{N,\vc\Delta}$, $\xi_{\infty,\vc\Delta}$ for the quantities in \eqref{MSKHamiltonian}, \eqref{cov_b}, \eqref{cov_c}, respectively. The only use of the balanced assumption \eqref{balanced} is in Proposition~\ref{key_claim}.

In the proposition below, inequality \eqref{rh8ef1z} provides a lower bound on $F_{N}(\vc\Delta)$ in terms of $F_{N}(\wt{\vc\Delta})$, where $\wt{\vc\Delta}$ is some other collection of interaction parameters.
By exchanging the roles of $\vc\Delta$ and $\wt{\vc\Delta}$, one obtains the two-sided inequality \eqref{rh9ef1z}, which shows Lipschitz continuity of the map $\vc\Delta\mapsto F_{N}(\vc\Delta)$ with respect to the metric
\eq{
(\vc\Delta,\wt{\vc\Delta})\mapsto\sup_{\vc x\in[-1,1]^\SSS}|\xi_{N,\vc\Delta}(\vc x)-\xi_{N,\wt{\vc\Delta}}(\vc x)|.
}
Inequalities \eqref{rh6ef1z} and \eqref{rh7ef1z} are the analogous statements but with $[0,1]^\SSS$ replacing $[-1,1]^\SSS$, which incurs an additional error that vanishes as $N\to\infty$.
We emphasize that \eqref{rh8ef1z} and \eqref{rh6ef1z} do not have absolute values; this is important in the proof of Theorem~\ref{thm_main}, specifically at \eqref{4h9fxx}.

\begin{proposition}[Comparing different interaction parameters] \label{prop_key}
Assume that both $(\vc \Delta,\vc\lambda_N)$ and $(\wt{\vc\Delta},\vc\lambda_N)$ satisfy \eqref{delta_decay_finite}.
For either $\sig\in\{\is,\sph\}$, we have
\begin{subequations} \label{without_pos}
\begin{align}
\label{rh8ef1z}
F_{N,\sig}(\wt{\vc\Delta},\vc\lambda_N)- F_{N,\sig}(\vc\Delta,\vc\lambda_N) 
&\leq \frac{1}{2}\Big[\xi_{N,\wt{\vc\Delta}}(\vc1)-\xi_{N,\vc\Delta}(\vc 1) \nonumber \\
&\phantom{\le\frac12\Big[}+ \sup_{\vc x\in[-1,1]^\SSS}\big(\xi_{N,\vc\Delta}(\vc x)-\xi_{N,\wt{\vc\Delta}}(\vc x)\big)\Big], \\
\label{rh9ef1z}
\text{and}\quad |F_{N,\sig}(\vc\Delta,\vc\lambda_N) - F_{N,\sig}(\wt{\vc\Delta},\vc\lambda_N)| &\leq \sup_{\vc x\in[-1,1]^\SSS} |\xi_{N,\vc\Delta}(\vc x)-\xi_{N,\wt{\vc\Delta}}(\vc x)|.
\end{align}
\end{subequations}
In the large-$N$ limit: if \eqref{converging_ratios}  holds, and both $(\vc\Delta,\vc\lambda)$ and $(\wt{\vc\Delta},\vc\lambda)$ satisfy \eqref{delta_decay}, then
\begin{subequations} \label{with_pos}
\begin{align}
\label{rh6ef1z}
\limsup_{N\to\infty}\big(F_{N,\sig}(\wt{\vc\Delta},\vc\lambda_N)-F_{N,\sig}(\vc\Delta,\vc\lambda_N)\big)
&\le\frac{1}{2}\Big[\xi_{\infty,\wt{\vc\Delta}}(\vc1)-\xi_{\infty,\vc\Delta}(\vc 1) \nonumber \\
&\phantom{\le\frac12\Big[}+\sup_{\vc x\in[0,1]^\SSS}\big(\xi_{\infty,\vc\Delta}(\vc x)-\xi_{\infty,\wt{\vc\Delta}}(\vc x)\big)\Big], \\
 \label{rh7ef1z}
\text{and}\quad\limsup_{N\to\infty}|F_{N,\sig}(\wt{\vc\Delta},\vc\lambda_N)-F_{N,\sig}(\vc\Delta,\vc\lambda_N)|
&\le\sup_{\vc x\in[0,1]^\SSS}|\xi_{\infty,\vc\Delta}(\vc x)-\xi_{\infty,\wt{\vc\Delta}}(\vc x)|.
\end{align}
\end{subequations}
\end{proposition}

Bootstrapping from \eqref{without_pos} to \eqref{with_pos} will require a multi-species version of \mbox{Talagrand}'s positivity principle, recalled below as Lemma~\ref{pos_principle}.
To state the lemma, we need to define a Gibbs measure:
given some (nonrandom) Hamiltonian $H\colon\Sigma_{N,\sig}\to\R$, let $G_{N,H}$ be the probability measure on $\Sigma_{N,\sig}$ defined by
\eq{ 
G_{N,H}(\dd\sigma) \coloneqq \frac{\exp H(\sigma)}{\int_{\Sigma_{N,\sig}}\exp H(\sigma')\, \mu_N(\dd\sigma')}\ \mu_N(\dd\sigma).
}
The strategy is to introduce a collection of perturbation parameters \eqref{perturbation_parameters} that vanish uniformly in the large-$N$ limit, yet are large enough to engage (an averaged form of) the Ghirlanda--Guerra identities. 
These identities lead to the positivity principle \eqref{pos_prin_eq}.
To help the reader parse the notation, we reiterate that $H_{N,\vc\delta_N}$ is the Hamiltonian \eqref{MSKHamiltonian} but with $\vc\delta_N$ in place of $\vc\Delta$.
Also, $\sigma^1$ and $\sigma^2$ denote independent samples from the Gibbs measure.

\begin{lemma}[Positivity principle]
\label{pos_principle}
Assume \eqref{converging_ratios} and \eqref{delta_decay}.
There is a collection of random variables 
\eeq{ \label{perturbation_parameters}
(\vc\delta_N)_{N\ge1} = (\delta_{s_1,\dots,s_p,N})_{s_1,\dots,s_p\in\SSS,\, p\ge1,\, N\ge1}
}
such that the following statements hold with $\vc\Delta_N$ defined by $\Delta^2_{s_1,\dots,s_p,N} \coloneqq \Delta^2_{s_1,\dots,s_p} + \delta^2_{s_1,\dots,s_p,N}$.
\begin{enumerate}[label=\textup{(\alph*)}]

\item \label{pos_principle_a} For each $N$, the collection $\vc\delta_N$ is independent of the disorder $(g_{i_1,\dots,i_p})_{1\le i_1,\dots,i_p\le N,\,p\ge1}$.

\item \label{pos_principle_b} $\delta^2_{s_1,\dots,s_p,N} \leq \frac{1}{3\cdot 4^{p-1}}N^{-1/2}$ for all $s_1,\dots,s_p\in\SSS$, $p\ge1$, $N\ge1$.

\item \label{pos_principle_c} ${\vc\Delta}_N$ satisfies \eqref{delta_decay_finite} for all large $N$.

\item \label{pos_principle_d} $\sup_{\vc x\in[-1,1]^\SSS}|\xi_{N,\vc\Delta}(\vc x)-\xi_{N,{\vc\Delta}_N}(\vc x)|\to0$ as $N\to\infty$.

\item \label{pos_principle_e} For either $\sig\in\{\is,\sph\}$ and any $\eps>0$, we have 
\eeq{ \label{pos_prin_eq}
\lim_{N\to\infty}\sup_{H}\E_{\vc\delta} \E_g G_{N,H+H_{N,\vc\delta_N}}^{\otimes 2}\Big(\bigcup_{s\in\SSS}\{R_s(\sigma^1,\sigma^2)\leq -\eps\}\Big)=0,
}
where $\E_{\vc\delta}$ denotes expectation over $\vc\delta_N$, $\E_g$ denotes expectation over the disorder, and the supremum is over
measurable functions $H\colon \Sigma_{N,\sig}\to\R$ satisfying
\eq{ 
\int_{\Sigma_{N,\sig}}\exp |H(\sigma)|\ \mu_N(\dd\sigma) < \infty.
}
\end{enumerate}
\end{lemma}

Our proof simply defers to \cite{bates_sohn22a}; while that paper treats the spherical case, the argument also works in the Ising case after trivial notational changes. This multi-species positivity principle was also used in \cite[Lemma~7.5]{issa24_arxiv}.

\begin{proof}[Proof of Lemma~\ref{pos_principle}]
In the notation of \cite[displays~(3.5)--(3.7)]{bates_sohn22a}, take 
\eq{
\delta^2_{s_1,\dots,s_p,N} \coloneqq c_N^2\sum_{q=1}^\infty\frac{u_{p,q}^2}{4^{p+q}}w_{q,s_1}\cdots w_{q,s_p},
}
where $c_N = N^{-1/4}$, $(u_{p,q})_{p,q\ge1}$ are i.i.d.\ $\mathsf{Uniform}[1,2]$ random variables selected independently of the disorder, and $(\vc w_q)_{q\geq1} = \big((w_{q,s})_{s\in\SSS}\big)_{q\ge1}$ is a deterministic dense subset of $[0,1]^\SSS$.
Part~\ref{pos_principle_a} is true by construction.
Part~\ref{pos_principle_b} is obtained by using the trivial bounds $u_{p,q}^2\le 4$ and $w_{q,s}\le1$, and then summing a geometric series.
Part~\ref{pos_principle_c} follows from part~\ref{pos_principle_b} together with \eqref{converging_ratios} and \eqref{delta_decay}.
Part~\ref{pos_principle_d} also follows from part~\ref{pos_principle_b}, since
\eq{
|\xi_{N,\vc\Delta}(\vc x)- \xi_{N,{\vc\Delta}_N}(\vc x)|
&\stackref{cov_b}{\leq} \sum_{p=1}^\infty \sum_{s_1,\dots,s_p\in\SSS}\delta^2_{s_1,\dots,s_p,N}\lambda_{s_1,N}\cdots\lambda_{s_p,N}\graybrace{|x_{s_1}\cdots x_{s_p}|}{\mathclap{\le 1 \text{ if $\vc x\in[-1,1]^\SSS$}}} \\
&\stackrefp{cov_b}{\leq} \sum_{p=1}^\infty \frac{N^{-1/2}}{3\cdot 4^{p-1}}\sum_{s_1,\dots,s_p\in\SSS}\lambda_{s_1,N}\cdots\lambda_{s_p,N} \\
&\stackrefp{cov_b}{=} N^{-1/2}\sum_{p=1}^\infty\frac{1}{3\cdot 4^{p-1}}\Big(\sum_{s\in\SSS}\lambda_{s,N}\Big)^p
= \frac{4}{9}N^{-1/2}.
}
Finally, part~\ref{pos_principle_e} is \cite[Lemma~3.3]{bates_sohn22a}.
\end{proof}

\begin{proof}[Proof of Proposition~\ref{prop_key}]
For ease of notation, we drop the symbols $\vc\lambda_N$ and $\sig$.
Since $\vc\Delta$ and $\wt{\vc\Delta}$ are logically interchangeable, \eqref{rh9ef1z} and \eqref{rh7ef1z} follow from \eqref{rh8ef1z} and \eqref{rh6ef1z} respectively.
We present the proofs of \eqref{rh8ef1z} and \eqref{rh6ef1z} together, using a parameter $\chi\in\{0,1\}$ to keep track of the minor differences.
If $\chi=0$, then no perturbation is performed on the Hamiltonian, and we will obtain \eqref{rh8ef1z}.
If instead $\chi=1$, then we include the random perturbation parameters $(\vc\delta_N)_{N\ge1}$ from Lemma~\ref{pos_principle}, and we will obtain \eqref{rh6ef1z}.
We assume throughout that $\vc\delta_N$ is independent of all Gaussian disorder.

As in Lemma~\ref{pos_principle}, define $\vc\Delta_N$ and $\wt{\vc\Delta}_N$ by
\eq{ 
\Delta^2_{s_1,\dots,s_p,N} \coloneqq \Delta^2_{s_1,\dots,s_p}+\chi\delta^2_{s_1,\dots,s_p,N}, \qquad
\wt\Delta^2_{s_1,\dots,s_p,N} \coloneqq \wt\Delta^2_{s_1,\dots,s_p}+\chi\delta^2_{s_1,\dots,s_p,N}.
}
To avoid any concerns about regularity, in the case $\chi=1$ we will assume $N$ is sufficiently large that both $(\vc\Delta_N,\vc\lambda_N)$ and $(\wt{\vc\Delta}_N,\vc\lambda_N)$ satisfy \eqref{delta_decay_finite}; this is possible by Lemma~\ref{pos_principle}\ref{pos_principle_c}.
Consider the following interpolating Hamiltonian: 
\begin{subequations} \label{H_facts}
\eeq{ 
\HH_{t}(\sigma)
\coloneqq \sqrt{1-t}\big[H_{N,\vc\Delta}(\sigma)+\chi H_{N,\vc\delta_N}(\sigma)\big] + \sqrt{t}\big[H_{N,\wt{\vc\Delta}}(\sigma)+\chi\wt H_{N,\vc\delta_N}(\sigma)\big], \quad t\in[0,1].
}
We assume the Gaussian disorders in the four Hamiltonians $H_{N,\vc\Delta}$, $H_{N,\vc\delta_N}$, $H_{N,\wt{\vc\Delta}}$, $\wt H_{N,\vc\delta_N}$ are independent.
Therefore,
\eeq{ 
H_{N,\vc\Delta}+\chi H_{N,\vc\delta_N} \stackrel{\mathrm{law}}{=} H_{N,\vc\Delta_N}, \qquad
H_{N,\wt{\vc\Delta}}+\chi \wt H_{N,\vc\delta_N} \stackrel{\mathrm{law}}{=} H_{N,\wt{\vc\Delta}_N}, \qquad \text{and} \\
\text{$H_{N,\vc\Delta}+\chi H_{N,\vc\delta_N}$ and $H_{N,\wt{\vc\Delta}}+\chi \wt H_{N,\vc\delta_N}$ are conditionally independent given $\vc\delta_N$.}
}
\end{subequations}
We will write $\E_{g,1}$ for expectation over the disorders in $H_{N,\vc\Delta}$ and $H_{N,\wt{\vc\Delta}}$, $\E_{g,2}$ for expectation over the disorders in $H_{N,\vc\delta_N}$ and $\wt H_{N,\vc\delta_N}$, and $\E_g$ for expectation over all disorders.
We will write $\E_{\vc\delta}$ for expectation over the randomness in $\vc\delta_N$.

Denote the free energy corresponding to $\HH_t$ by
\eeq{ \label{7vjxn_new}
\phi(t)\coloneqq\frac{1}{N} \mathbb{E}\log \ZZ_{t},\quad \text{where} \quad 
\ZZ_{t}\coloneqq \int_{\Sigma_N}\exp \HH_{t}(\sigma)\ \mu_N(\dd\sigma).
}
When $t=0$ or $t=1$ we recover the separate models:
\eeq{ \label{vxo4l0_new}
\phi(0) =  F_{N}(\vc\Delta_N), \qquad \phi(1) = F_{N}(\wt{\vc\Delta}_N).
}
At any intermediate value $t\in(0,1)$, we will compute the derivative of $\phi$. Let $\langle \cdot \rangle_t$ denote expectation with respect to $G_{N,\HH_t}^{\otimes2}$.
Differentiation of \eqref{7vjxn_new} gives
\begin{subequations} \label{jwl0v3_new}
\eeq{
\phi^\prime (t)=\frac{1}{N}\mathbb{E}\Big\langle \frac{\partial}{\partial t} \HH_{t}(\sigma^1) \Big\rangle_{t}
=\frac{1}{N}\E_{\vc\delta}\E_g\Big\langle \frac{\partial}{\partial t} \HH_{t}(\sigma^1) \Big\rangle_{t}.
}
Now apply Gaussian integration by parts \cite[Lemma~1.1]{panchenko13a} on the right-hand side:
\eeq{
\E_g\Big\langle\frac{\partial}{\partial t} \HH_{t}(\sigma^1) \Big\rangle_{t}
= \E_g\big\langle \CC(\sigma^1,\sigma^1)-\CC(\sigma^1,\sigma^2)\big\rangle_t,
}
\end{subequations}
where $\CC\colon\Sigma_N\times\Sigma_N\to\R$ is defined as
\eeq{ \label{7vkc3a_new}
\CC(\sigma,\tau) \coloneqq \E_g\Big(\frac{\partial}{\partial t}[\HH_{t}(\sigma)]\cdot\HH_{t}(\tau)\Big).
}
It is straightforward to calculate the right-hand side of \eqref{7vkc3a_new} using \eqref{H_facts}. This calculation results in
\eeq{ \label{8g2bfd3}
\CC(\sigma,\tau)
&\stackrefp{cov_a}{=}\frac{-\E_g[H_{N,\vc\Delta_N}(\sigma)H_{N,\vc\Delta_N}(\tau)]}{2} + \frac{\E_g[H_{N,\wt{\vc\Delta}_N}(\sigma)H_{N,\wt{\vc\Delta}_N}(\tau)]}{2} \\
&\stackref{cov_a}{=}\frac{N}{2}\big[\xi_{N,\wt{\vc\Delta}_N}\big(\vc R(\sigma,\tau)\big)-\xi_{N,\vc\Delta_N}\big(\vc R(\sigma,\tau)\big)\big],
}
Notice that when $\sigma = \tau$, we have $\vc R(\sigma,\sigma) = \vc 1$.
So using \eqref{8g2bfd3} within \eqref{jwl0v3_new} yields
\eeq{ \label{8t2gqd}
\phi'(t) &= \frac{1}{2}\Big[\xi_{N,\wt{\vc\Delta}_N}(\vc1)-\xi_{N,\vc\Delta_N}(\vc 1) + \E_{\vc\delta}\E_g\big\langle \xi_{N,\vc\Delta_N}\big(\vc R(\sigma^1,\sigma^2)\big)-\xi_{N,\wt{\vc\Delta}_N}(\vc R(\sigma^1,\sigma^2)\big)\big\rangle_t\Big].
}
If $\chi=0$ (so $\vc\Delta_N = \vc\Delta$ and $\wt{\vc\Delta}_N = \wt{\vc\Delta}$), then the following inequality is trivial from \eqref{8t2gqd}:
\eq{
\phi'(t) \le \frac{1}{2}\Big[\xi_{N,\wt{\vc\Delta}}(\vc1)-\xi_{N,\vc\Delta}(\vc 1)+\sup_{\vc x\in[-1,1]^\SSS}\big(\xi_{N,\vc\Delta}(\vc x)-\xi_{N,\wt{\vc\Delta}}(\vc x)\big)\Big].
}
In light of \eqref{vxo4l0_new}, this proves \eqref{rh8ef1z}.
If instead $\chi=1$, then we fix $\eps>0$ and observe the following consequence of \eqref{8t2gqd}:
\eq{
\phi'(t)
&\le \frac{1}{2}\Big[\xi_{N,\wt{\vc\Delta}_N}(\vc1)-\xi_{N,\vc\Delta_N}(\vc 1) \\
&\phantom{\le\frac12\Big[}+\E_{\vc\delta}\E_g\Big\langle\one_{\{\vc R(\sigma^1,\sigma^2)\in(-\eps,1]^\SSS\}}\Big(\xi_{N,\vc\Delta_N}\big(\vc R(\sigma^1,\sigma^2)\big)-\xi_{N,\wt{\vc\Delta}_N}(\vc R(\sigma^1,\sigma^2)\big)\Big)\Big\rangle_t \\
&\phantom{\le\frac12\Big[}+ \E_{\vc\delta}\E_g\Big\langle\one_{\{\vc R(\sigma^1,\sigma^2)\notin(-\eps,1]^\SSS\}}\Big( \xi_{N,\vc\Delta_N}\big(\vc R(\sigma^1,\sigma^2)\big)-\xi_{N,\wt{\vc\Delta}_N}(\vc R(\sigma^1,\sigma^2)\big)\Big)\Big\rangle_t\Big]\\
&\le \frac{1}{2}\Big[\xi_{N,\wt{\vc\Delta}_N}(\vc1)-\xi_{N,\vc\Delta_N}(\vc 1) \\
&\phantom{\le\frac12\Big[}+\E_{\vc\delta}\E_g\Big\langle\one_{\{\vc R(\sigma^1,\sigma^2)\in(-\eps,1]^\SSS\}}\Big\rangle_t\sup_{\vc x\in(-\eps,1]^\SSS}\big(\xi_{N,\vc\Delta_N}(\vc x)-\xi_{N,\wt{\vc\Delta}_N}(\vc x)\big) \\
&\phantom{\le\frac12\Big[}+ \E_{\vc\delta}\E_g\Big\langle\one_{\{\vc R(\sigma^1,\sigma^2)\notin(-\eps,1]^\SSS\}}\Big\rangle_t\sup_{\vc x\in[-1,1]^\SSS}\big(\xi_{N,\vc\Delta_N}(\vc x)-\xi_{N,\wt{\vc\Delta}_N}(\vc x)\big)\Big].
}
Since $\xi_{N,\vc\Delta_N}(\vc0) = \xi_{N,\wt{\vc\Delta}_N}(\vc0)=0$, both suprema seen above are nonnegative.
Therefore, we may use the trivial inequality $\one_{\{\vc R(\sigma^1,\sigma^2)\in(-\eps,1]^\SSS\}}\le 1$ to obtain a further upper bound:
\eeq{ \label{jdv93x}
\phi'(t)
&\le \frac{1}{2}\Big[\xi_{N,\wt{\vc\Delta}_N}(\vc1)-\xi_{N,\vc\Delta_N}(\vc 1) + \sup_{\vc x\in(-\eps,1]^\SSS}\big(\xi_{N,\vc\Delta_N}(\vc x)-\xi_{N,\wt{\vc\Delta}_N}(\vc x)\big) \\
&\phantom{\le} + \E_{\vc\delta}\E_g G_{N,\HH_t}^{\otimes2}\Big(\bigcup_{s\in\SSS}\{R_s(\sigma^1,\sigma^2)\leq -\eps\}\Big)\sup_{\vc x\in[-1,1]^\SSS}\big(\xi_{N,\vc\Delta_N}(\vc x)-\xi_{N,\wt{\vc\Delta}_N}(\vc x)\big)\Big].
}
Note that $\sqrt{1-t}H_{N,\vc\delta_N}+\sqrt{t}\wt H_{N,\vc\delta_N} \stackrel{\mathrm{law}}{=} H_{N,\vc\delta_N}$.
This justifies the inequality below:
\eq{
\E_{\vc\delta}\E_g G_{N,\HH_t}^{\otimes2}\Big(\bigcup_{s\in\SSS}\{R_s(\sigma^1,\sigma^2)\leq -\eps\}\Big)
&= \E_{g,1}\E_{\vc\delta}\E_{g,2}G_{N,\HH_t}^{\otimes2}\Big(\bigcup_{s\in\SSS}\{R_s(\sigma^1,\sigma^2)\leq -\eps\}\Big) \\
&\le \sup_H \E_{\vc\delta}\E_{g,2}G_{N,H+H_{N,\vc\delta_N}}^{\otimes2}\Big(\bigcup_{s\in\SSS}\{R_s(\sigma^1,\sigma^2)\leq -\eps\}\Big).
}
So by \eqref{pos_prin_eq}, we may assume $N$ is large enough that
\eq{
\E_{\vc\delta} \E_g G_{N,\HH_t}^{\otimes2}\Big(\bigcup_{s\in\SSS}\{R_s(\sigma^1,\sigma^2)\leq -\eps\}\Big) \le \eps \quad \text{for all $t\in(0,1)$}.
}
Inserting this inequality into \eqref{jdv93x} results in
\eq{
\sup_{t\in(0,1)} \phi'(t)
&\le \frac{1}{2}\Big[\xi_{N,\wt{\vc\Delta}_N}(\vc1)-\xi_{N,\vc\Delta_N}(\vc 1) + \sup_{\vc x\in(-\eps,1]^\SSS}\big(\xi_{N,\vc\Delta_N}(\vc x)-\xi_{N,\wt{\vc\Delta}_N}(\vc x)\big) \\
&\phantom{\le\frac12\Big[} + \eps\sup_{\vc x\in[-1,1]^\SSS}\big(\xi_{N,\vc\Delta_N}(\vc x)-\xi_{N,\wt{\vc\Delta}_N}(\vc x)\big)\Big].
} 
In light of \eqref{vxo4l0_new} and Lemma~\ref{pos_principle}\ref{pos_principle_d}, we infer from this derivative bound the following inequality:
\eq{
&\limsup_{N\to\infty}(F_{N}(\wt{\vc\Delta}_N)-F_{N}(\vc\Delta_N)) \\
&\le\limsup_{N\to\infty}\frac{1}{2}\Big[\xi_{N,\wt{\vc\Delta}}(\vc1)-\xi_{N,\vc\Delta}(\vc 1)+\sup_{\vc x\in(-\eps,1]^\SSS}\big(\xi_{N,\vc\Delta}(\vc x)-\xi_{N,\wt{\vc\Delta}}(\vc x)\big) \\
&\phantom{\le\limsup_{N\to\infty}\frac12\Big[}+ \eps \sup_{\vc x\in[-1,1]^\SSS}\big(\xi_{N,\vc\Delta}(\vc x)-\xi_{N,\wt{\vc\Delta}}(\vc x)\big)\Big].
}
Furthermore, we can use \eqref{rh9ef1z} together with Lemma~\hyperref[pos_principle_d]{\ref*{pos_principle}\ref*{pos_principle_d}} to see that
\eq{
\lim_{N\to\infty}|F_{N}(\vc\Delta)-F_{N}(\vc\Delta_N)| = 0 \qquad \text{and} \qquad
\lim_{N\to\infty}|F_{N}(\wt{\vc\Delta})-F_{N}(\wt{\vc\Delta}_N)| = 0.
}
Now \eqref{rh6ef1z} follows by using the two previous displays together with \eqref{uniform_cov_convergence}, and then sending $\eps\searrow0$.
\end{proof}
To go from positive temperature to zero temperature, we need the following standard lemma. 
In the statement below, $\alpha$ is a scalar acting as an inverse-temperature parameter.

\begin{lemma}[Ground state energy from free energy] \label{lem_free_to_GSE}
Assume \eqref{converging_ratios} and \eqref{delta_decay}.
For either $\sig\in\{\is,\sph\}$, we have
\eeq{ \label{free_to_GSE}
\lim_{\alpha\to\infty}\liminf_{N\to\infty}\frac{1}{\alpha }F_{N,\sig}(\alpha \vc \Delta, \vc \lambda_N)
&=\liminf_{N\to\infty}\GSE_{N,\sig}(\vc \Delta,\vc \lambda_N), \\
\text{and} \quad
\lim_{\alpha\to\infty}\limsup_{N\to\infty}\frac{1}{\alpha }F_{N,\sig}(\alpha \vc \Delta, \vc \lambda_N)
&=\limsup_{N\to\infty}\GSE_{N,\sig}(\vc \Delta,\vc \lambda_N).
}
\end{lemma}

\begin{proof}
We argue similarly to \cite[Lemma~8]{chen_sen17}.
For ease of notation, we drop the symbols $\vc\lambda_N$ and $\sig$.
For each $\eps\in (0,1)$, let $K_{N,\eps} \subseteq \Sigma_N$ be an $\eps\sqrt{N}$-net of $\Sigma_N$ (with respect to the $\ell^2$ norm) whose cardinality is 
\eeq{ \label{q6vm3}
\# K_{N,\eps}\leq (3/\eps)^N. 
}
(In the Ising case, one can take $K_{N,\eps}$ to be the entire configuration space $\{-1,1\}^N$.
In the spherical case, use \cite[Corollary~4.2.13]{vershynin18} to identify an $\eps$-net of the unit sphere $\mathbb{S}^{(\#\II_{s,N})-1}$ with cardinality at most $(3/\eps)^{\#\II_{s,N}}$, then scale these points by $\sqrt{\#\II_{s,N}}$ to produce an $\eps\sqrt{\#\II_{s,N}}$-net of $S_{\# \II_{s,N}}$, and finally let $K_{N,\eps}$ be the product of these nets over $s\in\SSS$.)
We trivially have
\eeq{ \label{ek9bc}
\max_{\sigma\in \Sigma_N}H_N(\sigma)
&\leq \max_{\sigma\in K_{N,\eps}}H_N(\sigma)+M_{N,\eps},
}
where
\eeq{ \label{MNeps_def}
M_{N,\eps} &\coloneqq \sup_{\sigma,\tau\in \Sigma_N,\, \|\sigma-\tau\|_2\leq \eps \sqrt{N}} |H_N(\sigma)-H_N(\tau)|.
}
Next we convert to the canonical metric
    \[
   d(\sigma,\tau)\coloneqq\sqrt{\E\big[\big(H_N(\sigma)-H_N(\tau)\big)^2\big]}
   \stackref{cov_a}{=}\sqrt{2N\big(\xi_N(\vc 1)-\xi_N(\vc R(\sigma,\tau))\big)},
    \]
which satisfies
\eq{
d(\sigma,\tau) 
&\le \sqrt{2N}\Big(\sup_{\vc x\in[-1,1]^\SSS}\|\nabla\xi_N(\vc x)\|_2^{1/2}\Big)\|\vc1-\vc R(\sigma,\tau)\|_2^{1/2} \\
&=\sqrt{2N}\Big(\sup_{\vc x\in[-1,1]^\SSS}\|\nabla\xi_N(\vc x)\|_2^{1/2}\Big)\bigg(\sum_{s\in\SSS}\Big(\frac{1}{2\cdot\#\II_{s,N}}\sum_{i\in\II_{s,N}}(\sigma_i-\tau_i)^2\Big)^2\bigg)^{1/4} \\
&\le\sqrt{2N}\Big(\sup_{\vc x\in[-1,1]^\SSS}\|\nabla\xi_N(\vc x)\|_2^{1/2}\Big)\bigg(\sum_{s\in\SSS}\frac{1}{2\cdot\#\II_{s,N}}\sum_{i\in\II_{s,N}}(\sigma_i-\tau_i)^2\bigg)^{1/2} \\
&= \Big(\sup_{\vc x\in[-1,1]^\SSS}\|\nabla\xi_N(\vc x)\|_2^{1/2}\Big)\bigg(\sum_{s\in\SSS}\frac{1}{\lambda_{s,N}}\sum_{i\in\II_{s,N}}(\sigma_i-\tau_i)^2\bigg)^{1/2} \\
&\le \Big(\sup_{\vc x\in[-1,1]^\SSS}\|\nabla\xi_N(\vc x)\|_2^{1/2}\Big)\Big(\inf_{s\in\SSS}\lambda_{s,N}\Big)^{-1}\|\sigma-\tau\|_2.
}
Under assumptions \eqref{converging_ratios} and \eqref{delta_decay}, the value of 
\eq{
\eta\coloneqq\lim_{N\to\infty}\Big(\sup_{\vc x\in[-1,1]^\SSS}\|\nabla\xi_N(\vc x)\|_2^{1/2}\Big)\Big(\inf_{s\in\SSS}\lambda_{s,N}\Big)^{-1}
= \Big(\sup_{\vc x\in[-1,1]^\SSS}\|\nabla\xi(\vc x)\|_2^{1/2}\Big)\Big(\inf_{s\in\SSS}\lambda_{s}\Big)^{-1}
}
is finite.
Without loss of generality, we assume $N$ is sufficiently large that $d(\sigma,\tau) \le 2\eta\|\sigma-\tau\|_2$ for all $\sigma,\tau\in\Sigma_N$, so that $K_{N,\eps}$ is a $2\eta\eps\sqrt{N}$-net of $\Sigma_N$ with respect to $d$, and
\eq{
M_{N,\eps} \le  \sup_{\sigma,\tau\in \Sigma_N,\, d(\sigma,\tau)\leq 2\eta\eps \sqrt{N}} |H_N(\sigma)-H_N(\tau)|.
}
It now follows from Dudley's bound \cite[Theorem~1.4.2]{talagrand21} that there exists a constant $L$ such that for every $N$ and $\eps$,
\eeq{ \label{ck4gf}
\E M_{N,\eps} 
&\stackrefp{q6vm3}{\le} L\cdot 2\eta \sqrt{N}\int_0^\eps\sqrt{\log\# K_{N,\eps'}}\ \dd\eps' \\
&\stackref{q6vm3}{\le} L\cdot 2\eta N\int_0^\eps\sqrt{\log(3/\eps')}\ \dd\eps' \\
&\stackrefp{q6vm3}{=} L\cdot 6\eta N\int_{3/\eps}^\infty \frac{\sqrt{\log u}}{u^2}\ \dd u \\
&\stackrefp{q6vm3}{\le} L\cdot 6\eta N\int_{3/\eps}^\infty \frac{1}{u^{3/2}}\ \dd u
= L\cdot 12\eta N\sqrt{\eps/3}.
}
We have thus controlled the second term on the right-hand side of \eqref{ek9bc}.

To control the other term, let $B_{\eps\sqrt{N}}(\sigma)$ denote the $\ell^2$ ball centered at $\sigma$ with radius $\eps\sqrt{N}$, and let $b_{N,\eps}\coloneqq\mu_N(B_{\eps\sqrt{N}}(\sigma)\cap\Sigma_N)$
(which is independent of $\sigma\in\Sigma_N$ by symmetry). 
Since $\Sigma_N = \bigcup_{\sigma \in K_{N,\eps}}(B_{\eps \sqrt{N}}(\sigma)\cap\Sigma_N)$ and $\mu_N(\Sigma_N) = 1$, we have $1\leq (\#K_{N,\eps})b_{N,\eps}$ and thus 
\eeq{ \label{q6vm4}
\frac{1}{b_{N,\eps}} \le \# K_{N,\eps}
\stackref{q6vm3}{\le} (3/\eps)^N.
}
We now have, for any $\alpha>0$,
\eq{
&\max_{\sigma\in K_{N,\eps}}\frac{H_N(\sigma)}{N}
= \frac{1}{N\alpha}\log\bigg(\max_{\sigma\in K_{N,\eps}}\frac{1}{b_{N,\eps}}\int_{B_{\eps\sqrt{N}}(\sigma)\cap\Sigma_N}\e^{\alpha H_N(\sigma)}\mu_N(\dd \tau)\bigg) \\
&\stackrefp{FNZN_def,q6vm4}{\le} \frac{1}{N\alpha}\log\bigg(\sum_{\sigma\in K_{N,\eps}}\frac{1}{b_{N,\eps}}\int_{B_{\eps\sqrt{N}}(\sigma)\cap\Sigma_N}\e^{\alpha H_N(\sigma)}\mu_N(\dd \tau)\bigg)\\
&\stackrefpp{MNeps_def}{FNZN_def,q6vm4}{\le} \frac{1}{N\alpha}\log\bigg(\sum_{\sigma\in K_{N,\eps}}\int_{B_{\eps\sqrt{N}}(\sigma)\cap\Sigma_N}\e^{\alpha H_N(\tau)}\mu_N(\dd \tau)\bigg)-\frac{1}{N\alpha} \log b_{N,\eps}+\frac{1}{N}M_{N,\eps}\\
&\stackrefp{FNZN_def,q6vm4}{\le} \frac{1}{N\alpha}\log\bigg(\int_{\Sigma_N}\e^{\alpha H_N(\tau)}\mu_N(\dd \tau)\bigg)+\frac{1}{N\alpha}\log(\# K_{N,\eps})-\frac{1}{N\alpha} \log b_{N,\eps}+\frac{1}{N}M_{N,\eps} \\
&\stackref{FNZN_def,q6vm4}{\le} \frac{1}{N\alpha}\log Z_{N}(\alpha\vc\Delta) + \frac{2}{\alpha}\log(3/\eps) + \frac{1}{N}M_{N,\eps}.
}
Now take expectation on both sides, use \eqref{ek9bc} and \eqref{ck4gf}, and send $N\to\infty$ to obtain
\eq{
\liminf_{N\to\infty}\GSE_N(\vc\Delta) 
&\le \liminf_{N\to\infty}\frac{F_N(\alpha\vc\Delta)}{\alpha} + \frac{2}{\alpha}\log(3/\eps) + L\cdot 12\eta\sqrt{\eps/3}, \\
\text{and}\quad
\limsup_{N\to\infty}\GSE_N(\vc\Delta) 
&\le \limsup_{N\to\infty}\frac{F_N(\alpha\vc\Delta)}{\alpha} + \frac{2}{\alpha}\log(3/\eps) + L\cdot 12\eta\sqrt{\eps/3}.
}
Next send $\alpha\to\infty$, followed by $\eps\to 0$:
\eeq{ \label{wjb98}
\liminf_{N\to\infty}\GSE_N(\vc\Delta) 
&\le \liminf_{\alpha\to0}\liminf_{N\to\infty}\frac{F_N(\alpha\vc\Delta)}{\alpha}, \\
\text{and} \quad
\limsup_{N\to\infty}\GSE_N(\vc\Delta) 
&\le \liminf_{\alpha\to0}\limsup_{N\to\infty}\frac{F_N(\alpha\vc\Delta)}{\alpha}.
}
On the other hand, since $\mu_N(\Sigma_N) = 1$, we trivially have (by comparing \eqref{FNZN_def} and \eqref{GSE_def})
\eq{
\frac{F_N(\alpha\vc\Delta)}{\alpha} \le\GSE_N(\vc\Delta) \quad \text{for all $\alpha>0$}. 
}
It follows that
\eeq{ \label{wjb99}
\limsup_{\alpha\to0}\liminf_{N\to\infty}\frac{F_N(\alpha\vc\Delta)}{\alpha}
&\le \liminf_{N\to\infty}\GSE_N(\vc\Delta) , \\
\text{and} \quad
\limsup_{\alpha\to0}\limsup_{N\to\infty}\frac{F_N(\alpha\vc\Delta)}{\alpha}
&\le \limsup_{N\to\infty}\GSE_N(\vc\Delta).
}
Combining \eqref{wjb98} and \eqref{wjb99} yields \eqref{free_to_GSE}.
\end{proof}

\begin{proof}[Proof of Theorem~\ref{thm_main}]
Using $\beta_p = \beta_p(\vc\Delta,\vc\lambda)$ from \eqref{beta_def}, we define interaction parameters $\Barabove{\vc\Delta}$ by
\eeq{ \label{rwe8gbx}
\Barabove{\Delta}^2_{s_1,\dots,s_p} \coloneqq \begin{cases} \beta_p^2/\lambda_s^{p-1} &\text{if $s_1=\cdots=s_p=s$} \\
 0 &\text{otherwise.}
\end{cases}
}
The associated Hamiltonian from \eqref{MSKHamiltonian} is
\eq{
H_{N,\barabove{\vc\Delta}}(\sigma) 
&= \sum_{p=1}^\infty\frac{1}{N^{(p-1)/2}}\sum_{s\in\SSS}\frac{\beta_p}{\lambda_{s}^{(p-1)/2}}\sum_{i_1,\dots,i_p\in\II_{s,N}}g_{i_1,\dots,i_p}\sigma_{i_1}\cdots\sigma_{i_p} \\
&= \sum_{s\in\SSS}\sum_{p=1}^\infty\beta_p\frac{\lambda_{s,N}^{(p-1)/2}}{\lambda_s^{(p-1)/2}}\cdot\frac{1}{(\#\II_{s,N})^{(p-1)/2}}\sum_{i_1,\dots,i_p\in\II_{s,N}}g_{i_1,\dots,i_p}\sigma_{i_1}\cdots\sigma_{i_p}.
}
Since there is no interaction between species, and $\Sigma_{N,\sig}$ is a product space, the free energy $F_{N,\sig}(\Barabove{\vc\Delta},\vc\lambda_N)$ from \eqref{FNZN_def} decouples into a sum of single-species free energies:
\eeq{ \label{7ygwnb}
F_{N,\sig}(\Barabove{\vc\Delta},\vc\lambda_N) = \frac{1}{N}\sum_{s\in\SSS}(\# \II_{s,N})F_{\# \II_{s,N},\sig}^{\single}\Bigg(\bigg(\beta_p\frac{\lambda_{s,N}^{(p-1)/2}}{\lambda_s^{(p-1)/2}}\bigg)_{p\ge1}\Bigg).
}
Meanwhile, the covariance functions $\xi_{N,\barabove{\vc\Delta}}$ and $\xi_{\infty,\barabove{\vc\Delta}}$ from \eqref{cov_b} and \eqref{cov_c} are given by
\eeq{ \label{wtxi_def}
\xi_{N,\barabove{\vc\Delta}}(\vc x) 
= \sum_{p=1}^\infty \sum_{s\in\SSS}\frac{\beta_p^2}{\lambda_s^{p-1}}\lambda_{s,N}^px_s^p, \qquad
\xi_{\infty,\barabove{\vc\Delta}}(\vc x)
= \sum_{p=1}^\infty \sum_{s\in\SSS}\beta_p^2\lambda_{s}x_s^p.
}
Note that $(\barabove{\vc\Delta},\vc\lambda)$ satisfies \eqref{delta_decay} since $(\vc\Delta,\vc\lambda)$ satisfies \eqref{delta_decay}:
\eeq{ \label{beta_decay}
\xi_{\infty,\barabove{\vc\Delta}}\big((1+\eps)\vc1)\big)
= \sum_{p=1}^\infty(1+\eps)^p\beta_p^2
&\stackref{beta_def}{=}\sum_{p=1}^\infty(1+\eps)^{p}\sum_{s_1,\dots,s_p\in\SSS}
\Delta^2_{s_1,\dots,s_p}\lambda_{s_1}\cdots\lambda_{s_p} \\
&\stackrefpp{cov_c}{beta_def}{=}\xi_{\infty,\vc\Delta}\big((1+\eps)\vc1\big)
\stackref{delta_decay}{<}\infty.
}
This will allow us to use Proposition~\ref{prop_key} with $\wt{\vc\Delta}=\Barabove{\vc\Delta}$.
The following result, which motivated Definition~\ref{def_balanced}, will make the outcome of that proposition useful in the balanced case.
A similar inequality appeared in \cite[Proposition~7.2]{issa24_arxiv}.

\begin{proposition}[Key inequality] \label{key_claim}
With $\Barabove{\vc\Delta}$ as in \eqref{rwe8gbx}, we have 
$\xi_{\infty,\vc\Delta}(\vc1)=\xi_{\infty,\barabove{\vc\Delta}}(\vc1)$ and $\xi_{\infty,\vc\Delta}(\vc x)\le\xi_{\infty,\barabove{\vc\Delta}}(\vc x)$ for every $\vc x\in[0,1]^\SSS$.
\end{proposition}

\begin{proofclaim}
The equality $\xi_{\infty,\vc\Delta}(\vc 1)=\xi_{\infty,\barabove{\vc\Delta}}(\vc 1)$ is by construction: simply take $\eps=0$ in \eqref{beta_decay}.
Next we prove the inequality $\xi_{\infty,\vc\Delta}(\vc x)\le\xi_{\infty,\barabove{\vc\Delta}}(\vc x)$.
First consider $p\ge2$.
The AM--GM inequality gives
\eq{
x_{s_1}\cdots x_{s_p}
\leq \frac{1}{p}\sum_{i=1}^p x_{s_i}^p \quad \text{for all $x_{s_1},\dots,x_{s_p}\ge0$}.
}
Multiply both sides by $\Delta^2_{s_1,\dots,s_p}\lambda_{s_1}\cdots\lambda_{s_p}$, and then sum over $s_1,\dots,s_p$:
\eeq{ \label{94dgxn3}
\sum_{s_1,\dots,s_p\in\SSS}\Delta^2_{s_1,\dots,s_p} \lambda_{s_1}\cdots\lambda_{s_p}x_{s_1}\cdots x_{s_p}
&\le \sum_{s_1,\dots,s_p\in\SSS}\Delta^2_{s_1,\dots,s_p}\lambda_{s_1}\cdots\lambda_{s_p}\frac{1}{p}\sum_{i=1}^p x_{s_i}^p \\
&=\frac{1}{p}\sum_{i=1}^p\sum_{s_i\in \SSS}\la_{s_i}x_{s_i}^p\sum_{(s_j)_{j\neq i}\in \SSS^{p-1}}\Delta^2_{s_1,\ldots, s_p} \prod_{j\neq i}\la_{s_j} \\
&=\sum_{s_1\in \SSS}\la_{s_1}x_{s_1}^p\sum_{s_2,\dots,s_p\in\SSS}\Delta^2_{s_1,s_2,\ldots, s_p}\lambda_{s_2}\cdots\lambda_{s_p},
}
where the final equality is due to symmetry of the map $(s_1,\dots,s_p)\mapsto\Delta^2_{s_1,\dots,s_p}$.
By the second line of the balanced assumption \eqref{balanced}, the value of the inner sum on the final line of \eqref{94dgxn3} is equal to some constant $C_p$ not depending on $s_1$.
On the other hand, we have
\eq{
C_p = \sum_{s_1\in\SSS}\lambda_{s_1}C_p = \sum_{s_1\in\SSS}\lambda_{s_1}\sum_{s_2,\dots,s_p\in\SSS}\Delta^2_{s_1,s_2,\dots,s_p}\lambda_{s_2}\cdots \lambda_{s_p}
\stackref{beta_def}{=} \beta_p^2.
}
Now \eqref{94dgxn3} reads as
\eeq{ \label{94dgxn4}
\sum_{s_1,\dots,s_p\in\SSS}\Delta^2_{s_1,\dots,s_p} \lambda_{s_1}\cdots\lambda_{s_p}x_{s_1}\cdots x_{s_p}
\leq \beta_p^2\sum_{s\in\SSS}\lambda_{s}x_s^p.
}
This inequality also holds for $p=1$, since the first line of \eqref{balanced} guarantees $\Delta_s^2 = \beta_1^2$ for every $s\in\SSS$.
So we may sum \eqref{94dgxn4} over all $p\ge1$: the left-hand side yields $\xi_{N,\vc\Delta}(\vc x)$ from \eqref{cov_b}, while the right-hand side yields $\xi_{\infty,\barabove{\vc\Delta}}(\vc x)$ from \eqref{wtxi_def}.
\end{proofclaim}

Now we invoke Proposition~\ref{prop_key} with $\wt{\vc\Delta}=\Barabove{\vc\Delta}$. 
As $N\to\infty$, we have
\eeq{ \label{4h9fxx}
&\liminf_{N\to\infty} F_{N,\sig}(\vc\Delta,\vc\lambda_N) \\
&\stackref{rh6ef1z}{\ge} \liminf_{N\to\infty} F_{N,\sig}(\Barabove{\vc\Delta},\vc\lambda_N) -\frac{1}{2}\bigg(\underbrace{\xi_{\infty,\barabove{\vc\Delta}}(\vc 1)-\xi_{\infty,\vc\Delta}(\vc 1) + \sup_{\vc x\in[0,1]^\SSS}\big(\xi_{\infty,\vc\Delta}(\vc x)-\xi_{\infty,\barabove{\vc\Delta}}(\vc x)\big)}_{\text{$\le0$ by Proposition~\ref{key_claim}}}\bigg) \\
&\stackref{7ygwnb}{\geq} \liminf_{N\to\infty}\sum_{s\in\SSS}\lambda_{s,N}F_{\#\II_{s,N},\sig}^{\single}\Bigg(\bigg(\beta_p\frac{\lambda_{s,N}^{(p-1)/2}}{\lambda_s^{(p-1)/2}}\bigg)_{p\ge1}\Bigg).
}
Next we invoke Proposition~\ref{prop_key} once more, but in the single-species case.
Namely, if we take $\SSS = \{s\}$, $\Delta_{s,\dots,s} = \beta_p\lambda_{s,N}^{(p-1)/2}\lambda_s^{-(p-1)/2}$, and $\wt\Delta_{s,\dots,s} = \beta_p$, then \eqref{rh9ef1z} implies
\eq{
\bigg|F_{\#\II_{s,N},\sig}^{\single}\Bigg(\bigg(\beta_p\frac{\lambda_{s,N}^{(p-1)/2}}{\lambda_s^{(p-1)/2}}\bigg)_{p\ge1}\Bigg)-F_{\#\II_{s,N},\sig}^{\single}(\bbeta)\bigg|
\le \sum_{p=1}^\infty\beta_p^2\Big|\Big(\frac{\lambda_{s,N}}{\lambda_s}\Big)^{p-1}-1\Big|.
}
By \eqref{converging_ratios} and \eqref{beta_decay}, we can use dominated convergence to conclude that the right-hand side tends to $0$ as $N\to\infty$.
Therefore, the two previous displays together imply 
\eq{
\liminf_{N\to\infty} F_{N,\sig}(\vc\Delta,\vc\lambda_N) \geq \sum_{s\in\SSS}\lambda_s \liminf_{N\to\infty} F_{N,\sig}^{\single}\big(\bbeta).
}
Since $\sum_{s\in\SSS}\lambda_s = 1$ and $F_{N,\sig}^\single(\bbeta)$ converges as $N\to\infty$ (see Remark~\ref{rmk_conv}), we have proved the desired inequality \eqref{eq:thm:msk}.

To extend the result to zero temperature, we invoke Lemma~\ref{lem_free_to_GSE} twice and use the positive-temperature result in between:
\eq{
\liminf_{N\to\infty} \GSE_{N,\sig}(\vc\Delta,\vc\lambda_N)
&\stackref{free_to_GSE}{=}\lim_{\alpha\to\infty}\liminf_{N\to\infty} \frac{1}{\alpha}F_{N,\sig}(\alpha\vc\Delta,\vc\lambda_N) \\
&\stackref{eq:thm:msk}{\ge}\lim_{\alpha\to\infty}\lim_{N\to\infty} \frac{1}{\alpha}F_{N,\sig}^\single(\alpha\bbeta) \\
&\stackref{free_to_GSE}{=}\lim_{N\to\infty} \GSE_{N,\sig}^\single(\bbeta).
}
We have thus proved \eqref{GSE_lower}.
\end{proof}

\section{Proofs of applications: matching upper bounds} \label{sec_proof_applications}

\subsection{Convex case} \label{subsec_upper_bound}

In this section, we prove Corollary~\ref{cor_convex}.
We begin by defining the Parisi functional for both the Ising case $\Sigma_{N,\is}=\{-1,+1\}^N$ and the spherical case $\Sigma_{N,\sph}=\otimes_{s\in \SSS}S_{\#\II_{s,N}}$. 

Given a covariance function $\xi\colon\R^\SSS\to\R$ as in \eqref{cov_c}, 
define the related functions
\eeq{ \label{deriv_def}
\xi^{s}(\bx)\coloneqq\frac{1}{\la_s}\frac{\partial \xi}{\partial x_{s}}(\bx),\qquad \theta(\bx)\coloneqq\Big(\sum_{s\in \SSS}x_s\frac{\partial \xi}{\partial x_{s}}(\bx)\Big) -\xi(\bx).
}
Both the Ising Parisi functional $\PPP_{\is}$ and the spherical Parisi functional $\PPP_{\sph}$ take as inputs any sequence $m=(m_{j})_{j\in\{0,\dots,k\}}$ and any array $\bq = (q_{j,s})_{j\in\{0,\dots,k+1\},\, s\in \SSS}$ of the following form:
\begin{equation}\label{eq:fop:multi}
\begin{split}
     &0 = m_{0} < m_1 < \dots < m_{k} = 1,\\
     &0 = q_{0,s} \leq q_{1,s} \leq \dots \leq q_{k+1,s} = 1 \quad \text{for each $s \in \mathscr{S}$}.
\end{split}
\end{equation}
Here $k$ can be any positive integer.
Given $\bq$, we write $\bq_j = (q_{j,s})_{s\in \SSS}$. 

In the Ising case, let $\eta_0,\dots,\eta_{k}$ be independent standard normal random variables, and then define the following random quantity for each species $s\in\SSS$:
\begin{equation*}
    X_{k+1,s}\coloneqq\log\cosh \Big(\sum_{j=0}^{k}\eta_{j}\sqrt{\xi^s(\bq_{j+1})-\xi^s(\bq_{j})}\,\Big).
\end{equation*} 
Using $\E_{j}$ to denote expectation with respect to $\eta_{j}$, we recursively define 
\eeq{\label{iteratingX}
X_{j,s} &\coloneqq \frac{1} {m_{j}} \log \E_{j} \exp(m_{j} X_{j+1,s})  \quad \text{for $j\in\{1,\dots,k\}$,} \\
X_{0,s} &\coloneqq \E_0 X_{1,s}.
} 
Finally, the Ising Parisi functional is given by
\eeq{\label{Parisifunctional}
\mathscr{P}_{\is}(m,\bq) \coloneqq\sum_{s \in \mathscr{S}} \lambda_s X_{0,s} - \frac{1} {2}\sum_{j=1}^{k} m_{j}\big(\theta(\bq_{j+1})-\theta(\bq_{j})\big).
}

\begin{remark}
Our Parisi functional \eqref{Parisifunctional} does not contain an additional summand of $\log 2$ that appears in \cite{panchenko15}.
This is because our partition function $Z_{N,\is}$ in \eqref{FNZN_def} includes a factor of $2^{-N}$ from the presence of the uniform probability measure $\mu_{N,\is}$ on $\{-1,+1\}^N$.
\QED
\end{remark}

In the spherical case,
for each species we define an auxiliary sequence $(d_{j,s})_{j\in\{1,\dots,k+1\}}$ by
\eq{
d_{j,s} &\coloneqq\sum_{j' = j}^{k}m_{j'}\big(\xi^s(\bq_{j'+1})-\xi^s(\bq_{j'})\big) \quad \text{for $j\in\{1,\dots,k\}$}, \qquad \text{and} \qquad
d_{k+1,s} \coloneqq 0.
}
Given any $b > d_{1,s}$, we define the number
\eeq{ \label{A_def}
X_s(b) \coloneqq b - 1 - \log b + \frac{\xi^s(\bq_1)}{b-d_{1,s}} + \sum_{j=1}^{k}\frac{1}{m_j}\log\frac{b-d_{j+1,s}}{b-d_{j,s}}.
}
Finally, the spherical Parisi functional is given by
\eeq{ \label{sph_par_func}
\PPP_{\sph}(m,\bq)
\coloneqq\sum_{s\in\SSS}\lambda_s\inf_{b_s>d_{1,s}}X_s(b_s)- \frac{1} {2}\sum_{j=1}^{k} m_{j}\big(\theta(\bq_{j+1})-\theta(\bq_{j})\big).
}

\begin{proposition}[{Broken replica symmetry bound, \cite[Theorem~1]{panchenko15}, \cite[Proposition~3.1]{bates_sohn22a}}]\label{prop:guerra} 
Assume \eqref{converging_ratios}, \eqref{delta_decay}, and that $\xi$ is convex on $[0,1]^\SSS$. 
Then for either $\sig\in \{\is,\sph\}$, we have
     \begin{equation}\label{eq:prop:guerra}
         \limsup_{N\to\infty} F_{N,\sig}(\vc\Delta,\vc\lambda_N)\leq \inf_{m,\bq}\PPP_{\sig}(m,\bq),
     \end{equation}
where the infimum is taken over all finite sequences $m$ and $\vc q$ of the form \eqref{eq:fop:multi}.
\end{proposition}

\begin{remark}[Clarifications for Proposition~\ref{prop:guerra}]
In the spherical case, \cite[Proposition~3.1]{bates_sohn22a} is slightly more general than what is presented above.
Namely, \cite{bates_sohn22a} defines $\PPP_\sph$ on $\vc\lambda$-admissible pairs $(\zeta,\Phi)$, meaning $\zeta$ is a Borel probability measure on $[0,1]$, and $\Phi$ (the ``synchronization map'') is a continuous function $[0,1]\to[0,1]^\SSS$ that is nondecreasing in each coordinate, and satisfies $\sum_{s\in\SSS}\lambda_s\Phi_s(x) = x$ for all $x\in[0,1]$.
What is seen in \eqref{eq:prop:guerra} is the special case of the purely atomic measure
\eq{
\zeta = \sum_{j=1}^k m_j\delta_{q_j}, \quad \text{where} \quad q_j = \sum_{s\in\SSS}\lambda_s q_{j,s},
}
paired with any $\Phi$ such that $\Phi(q_j) = \vc q_j$ for each $j\in\{0,\dots,k+1\}$, e.g.\ $\Phi(x)$ is defined by linear interpolation whenever $q_j < x < q_{j+1}$.
In this special case, \eqref{A_def} can be found in \cite[display~(2.41)]{bates_sohn22a}.
The inclusion of general measures does not change the infimum in \eqref{eq:prop:guerra}, since a general measure can be approximated by discrete ones, and the Parisi functional is continuous \cite[Theorem~1.5]{bates_sohn22a}.

In the Ising case, \cite[Theorem~1]{panchenko15} is slightly less general than what is presented above.
Namely, \cite{panchenko15} only considers the SK model: $\Delta^2_{s_1,\dots,s_p} = 0$ for all $p\ne 2$.
Nevertheless, the upper bound free the mixed $p$-spin model can be handled in the same way (i.e.\ using Guerra interpolation) under the stated convexity assumption.  
The only nontrivial modification is to justify why convexity on $[0,1]^\SSS$ is sufficient (rather than all of $[-1,1]^\SSS$).
This is done using a multi-species version of Talagrand's positivity principle (appearing here as Lemma~\ref{pos_principle}), just as in \cite[Section~3.3]{bates_sohn22a}.
In fact, the only part of \cite[Section~3.3]{bates_sohn22a} that is specific to the spherical case is the formula \cite[display~(2.47)]{bates_sohn22a}, which is replaced in the Ising case by \cite[display~(21)]{panchenko15} suitably adapted for mixed $p$-spins.

A final technical detail is that \cite{panchenko15} only considers cases where $k\geq2$, $q_{1,s} = 0$, and $q_{k,s}=1$.
In other words, it is assumed that the measure $\zeta$ places positive mass on both $0$ and $1$.
The computations are not complicated very much by removing this assumption (e.g.\ see \cite[display~(14.84)]{talagrand11b}).
In any case this distinction does not change the value of the infimum in \eqref{eq:prop:guerra}, again by approximation together with continuity of the Parisi functional.
\QED
\end{remark}

The key observation to prove Corollary~\ref{cor_convex} from Proposition~\ref{prop:guerra} is that when we insist $q_{j,s}=q_{j,s'}$ for all $j\in\{0,\dots,k+1\}$ and $s,s'\in \SSS$, the right-hand side of \eqref{eq:prop:guerra} reduces to the single-species Parisi formula.
To state this precisely, we define the single-species Parisi functional as follows. 
Suppose that $\xi$ is balanced (Definition~\ref{def_balanced}), and consider the corresponding single-species covariance function $\xi^{\single}\colon[-1,1]\to\R$ given by
\eeq{ \label{xi_sin_def}
\xi^\single(x)
\coloneqq\sum_{p=1}^\infty\beta_p^2 x^p
\stackref{beta_def}{=} \sum_{p=1}^\infty\Big(\sum_{s_1,\dots,s_p\in\SSS}
\Delta^2_{s_1,\dots,s_p}\lambda_{s_1}\cdots\lambda_{s_p}\Big)x^p
\stackref{cov_c}{=}\xi(x\cdot\vc1).
}
For $\sig\in \{\is,\sph\}$, let $\PPP_{\sig}^{\single}$ denote the special case of $\PPP_{\sig}$ when $\#\SSS=1$ and the covariance function is given by $\xi^{\single}(x)$.
By the references in Remark~\ref{rmk_spar} (specifically \cite[Theorem~1]{panchenko14} in the Ising case, and \cite[Theorem~1.1]{chen13} in the spherical case), we have
\begin{equation}\label{eq:single:parisi}
    \lim_{N\to\infty} F_{N,\sx}^\single(\bbeta)
    = \inf_{m,q}\PPP_{\sx}^\single(m,q).
\end{equation}
The key calculation is the following:
\begin{lemma}[Recovering the single-species functional]\label{lem:parisi:ftl:reduce}
Assume \eqref{delta_decay} and \eqref{balanced}.
Given a sequence $q = (0 = q_0 \le q_1 \le \cdots \le q_k \le q_{k+1} = 1)$, let $\wh{\vc q} = (\wh{\vc q}_j)_{j\in\{0,\dots,k+1\}}$ be given by $\wh{\vc q}_j = q_j\cdot\vc 1$ for each $j$.
For either $\sx\in \{\is,\sph\}$ and any sequence $m = (0 = m_0 < m_1 < \cdots < m_k = 1)$, we have
\eeq{ \label{dk4n1}
    \PPP_{\sx}(m,\wh{\vc q})=\PPP^\single_{\sx}(m,q).
}
\end{lemma}

\begin{proof}
   We first claim that $\xi^{s}(x\cdot \bone)=\big(\xi^{\single}\big)^\prime(x)$ for all $x\in[-1,1]$ and $s\in \SSS$. 
   Indeed, by a direct calculation,
\eeq{ \label{same_deriv}
 \xi^{s}(x\cdot \bone)
 &\stackrefpp{deriv_def}{balanced}{=}
 \Delta_s^2 + \sum_{p=2}^{\infty}\sum_{s_2,\ldots, s_p\in \SSS}
 p\Delta^2_{s,s_2,\ldots, s_p}\la_{s_2}\cdots \la_{s_p}\cdot x^{p-1} \\
 &\stackref{balanced}{=}\sum_{p=1}^\infty \sum_{s_1,s_2,\dots,s_p\in\SSS}p\Delta^2_{s_1,s_2,\ldots, s_p}\la_{s_1}\la_{s_2}\cdots \la_{s_p}\cdot x^{p-1}
 \stackref{xi_sin_def}{=}(\xi^\single)^\prime(x).
}
We now have
\eeq{ \label{theta_same}
 \theta(x\cdot\bone)
 \stackref{deriv_def}{=}\Big(x\sum_{s\in \SSS} \la_s \xi^{s}(x\cdot\bone)\Big) -\xi 
(x\cdot \bone)
\stackref{same_deriv,xi_sin_def}{=}x(\xi^\single)^\prime(x)-\xi^\single(x).
}
Moreover, since $\xi^s(\wh{\vc q}_j) = \xi^{s}(q_j\cdot\bone)=\big(\xi^\single\big)^\prime(q_j)$, we have the following: 
\begin{itemize}
    \item In the Ising case, the value of $X_{0,s}$ in \eqref{iteratingX} is the same across species $s\in \SSS$.
    Together with \eqref{theta_same}, this implies that the functional  \eqref{Parisifunctional} satisfies $\PPP_{\is}(m, \wh{\vc q})=\PPP_{\is}^{\single}(m, q)$.
    \item In the spherical case, the value of $X_s(b)$ in \eqref{A_def} is the same across species $s\in\SSS$.
    Together with \eqref{theta_same}, this implies that the functional  \eqref{sph_par_func} satisfies $\PPP_{\sph}(m, \wh{\vc q})=\PPP_{\sph}^{\single}(m, q)$.  \qedhere  
\end{itemize}
\end{proof}
\begin{proof}[Proof of Corollary~\ref{cor_convex}]
Using the notation $q$ and $\wh{\vc q}$ from Lemma~\ref{lem:parisi:ftl:reduce}, we have
\eeq{ \label{d7blp}
    \limsup_{N\to\infty} F_{N,\sx}(\vc\Delta,\vc\lambda_N)
    \stackref{eq:prop:guerra}{\leq} \inf_{m,q}\PPP_{\sx}(m,\wh{\vc q})
    \stackref{dk4n1}{=}\inf_{m,q}\PPP_{\sx}^\single(m,q)
    \stackref{eq:single:parisi}{=}
    \lim_{N\to\infty}F_{N,\sig}^\single(\bbeta).
}
Combining this upper bound with the lower bound \eqref{eq:thm:msk} yields \eqref{eq_msk2_a}.

To extend the result to zero temperature, we invoke Lemma~\ref{lem_free_to_GSE} twice and use the positive-temperature result in between:
\eq{
\limsup_{N\to\infty} \GSE_{N,\sig}(\vc\Delta,\vc\lambda_N)
&\stackref{free_to_GSE}{=}\lim_{\alpha\to\infty}\limsup_{N\to\infty} \frac{1}{\alpha}F_{N,\sig}(\alpha\vc\Delta,\vc\lambda_N) \\
&\stackref{d7blp}{\le}\lim_{\alpha\to\infty}\limsup_{N\to\infty} \frac{1}{\alpha}F_{N,\sig}^\single(\alpha\bbeta) \\
&\stackref{free_to_GSE}{=}\lim_{N\to\infty} \GSE_{N,\sig}^\single(\bbeta).
}
Combining this upper bound with the lower bound \eqref{GSE_lower} yields \eqref{eq_msk2_b}.
\end{proof}

\subsection{Ground state energy of balanced pure bipartite spherical model} \label{subsec_proof_bipartite}

\begin{proof}[Proof of Corollary~\ref{cor_bpbs}]
We assume $\vc\lambda$ and $\vc\Delta$ are given by \eqref{wfe78}.
It follows that the balanced condition \eqref{balanced} holds; see Example~\ref{ex_model}\ref{ex_model_3}.
Therefore, we may apply Theorem~\ref{thm_main}, which gives
\eeq{ \label{k39bs}
\liminf_{N\to\infty}\GSE_{N,\sph}(\vc\Delta,\vc\lambda_N)
\ge \lim_{N\to\infty}\GSE_{N,\sph}^\single(\bbeta),
}
where $\bbeta$ is given by \eqref{beta_def}.
Because this $\bbeta$ is pure $(p+q)$-spin (i.e.\ $\beta_r = 0$ for all $r\neq p+q$) with $\beta_{p+q}^2 = \xi(1,1) \stackref{wfe78_a}{=} \beta^2$, the limit statement \eqref{3r8gx} guarantees
\eeq{ \label{jd2bc}
\lim_{N\to\infty}\GSE_{N,\sph}^\single(\bbeta)
= \sqrt{\beta^2}E_0(p+q).
}
On the other hand, under the additional assumption \eqref{volume_assumption}, it was shown in \cite[Corollary~2.9]{mckenna24} (note the normalization $\xi(1,1) = 1$ from \cite[page~640]{mckenna24}, whereas we have $\xi(1,1)=\beta^2$) that for any $\eps>0$, there exist constants $C,c>0$ such that
\eq{
\P\Big(\max_{\sigma\in\Sigma_{N,\sph}}\frac{H_N(\sigma)}{N} \ge \sqrt{\beta^2}E_0(p+q) + \eps\Big) \le C\exp(-cN).
}
It follows from Borel--Cantelli (and sending $\eps\searrow0$) that
\eq{
\limsup_{n\to\infty}\max_{\sigma\in\Sigma_{N,\sph}}\frac{H_N(\sigma)}{N} \le \sqrt{\beta^2}E_0(p+q)
\quad \mathrm{a.s.}
}
By concentration \eqref{borell_tis} and another application of Borel--Cantelli, we also have
\eq{
\lim_{N\to\infty}\Big|\max_{\sigma\in\Sigma_{N,\sph}}\frac{H_N(\sigma)}{N} - \GSE_{N,\sph}(\vc\Delta,\vc\lambda_N)\Big| = 0 \quad \mathrm{a.s.}
}
The two previous displays together show
\eeq{ \label{23kbct}
\limsup_{n\to\infty}\GSE_{N,\sph}(\vc\Delta,\vc\lambda_N) \le \sqrt{\beta^2}E_0(p+q).
}
The proof is completed by combining \eqref{k39bs}, \eqref{jd2bc}, and \eqref{23kbct}.
\end{proof}

\subsection{Injective norm of Gaussian random tensors} \label{subsec_proof_inj_norm}
In this section we prove Corollary~\ref{cor_inj_norm}.
Recall that the tensor $T=(T_{i_1,\ldots, i_p})_{1\leq i_1,\ldots i_p\leq d}$ consists of independent standard normal random variables.
The following lemma relates the injective norm of $T$ to the ground state energy of a certain multi-species spin glass.

\begin{lemma}[Injective norm is a ground state energy] \label{lem_translate}
Given $p,d\in\Z_{\ge1}$, let $N = p d$ and consider the multi-species model \eqref{MSKHamiltonian} with $\SSS = \{1,\dots,p\}$, $\# \II_{s,N} = d$ for each $s\in\SSS$, and 
\eeq{ \label{73hb}
\Delta^2_{s_1,\dots,s_q} = 0 \quad \text{whenever $q\ne p$}, \qquad
\Delta^2_{s_1,\dots,s_p} = \begin{cases}
\frac{p^{p+1}}{p!} &\text{if $s_1,\dots,s_p$ are distinct} \\
0 &\text{else.}
\end{cases}
}
Then the covariance function is
\eeq{ \label{37gcb}
\xi_N(x_1,\dots,x_p) = px_1\cdots x_p, 
}
and we have the following distributional equality:
\eeq{ \label{38gcb}
\frac{\|T\|_\inj}{\sqrt{d}} \stackrel{\mathrm{law}}{=} \max_{\sigma\in\Sigma_{N,\sph}}\frac{H_{N}(\sigma)}{N}.
}
\end{lemma}

\begin{proof}
The first assertion \eqref{37gcb} is immediate from the definition \eqref{cov_b} and the fact that $\lambda_{s,N} = 1/p$ for each $s\in\{1,\dots,p\}$.
Since $\# \II_{s,N} = d$ for each $s$, our configuration space is the following $p$-fold product: 
\eq{
\Sigma_{N,\sph} = \sqrt{d}\,\S^{d-1}\times \cdots\times \sqrt{d}\,\S^{d-1}.
}
So let us write a given configurations as $\sigma = (\sigma^{(1)},\dots,\sigma^{(p)})$, where $\sigma^{(s)}\in \sqrt{d}\,\S^{d-1}$ for each $s\in\{1,\dots,p\}$.
Using this notation, define $\tilde H_N\colon\Sigma_{N,\sph}\to\R$ by
\eq{
\frac{\tilde H_N(\sigma)}{N} = \frac{1}{\sqrt{d}}\sum_{i_1,\dots,i_p=1}^d T_{i_1,\dots,i_p}\frac{\sigma^{(1)}_{i_1}}{\sqrt{d}}\cdots\frac{\sigma^{(p)}_{i_p}}{\sqrt{d}}.
}
By definition of the injective norm \eqref{def_inj_norm}, we have
\eq{
 \max_{\sigma\in\Sigma_{N,\sph}}\frac{\tilde H_N(\sigma)}{N} = \frac{\|T\|_\inj}{\sqrt{d}}.
}
To obtain \eqref{38gcb}, it suffices to show that $(\tilde H_N(\sigma))_{\sigma\in\Sigma_{N,\sph}}$ has the same joint distribution as $( H_N(\sigma))_{\sigma\in\Sigma_{N,\sph}}$.
Since these are centered Gaussian processes with continuous sample paths, we just need that their covariances agree.
That is, we need to check
\eeq{ \label{ug336}
\E[\tilde H_N(\sigma)\tilde H_N(\tau)] = \E[H_N(\sigma)H_N(\tau)] \quad \text{for all $\sigma,\tau\in\Sigma_{N,\sph}$}.
}
Since the entries of $T$ are independent with mean $0$ and variance $1$, we have
\eq{
\text{LHS of \eqref{ug336}} 
&=\frac{N^2}{d}\sum_{i_1,\dots,i_p=1}^d \frac{\sigma^{(1)}_{i_1}}{\sqrt{d}}\cdots\frac{\sigma^{(p)}_{i_p}}{\sqrt{d}}
\cdot\frac{\tau^{(1)}_{i_1}}{\sqrt{d}}\cdots\frac{\tau^{(p)}_{i_p}}{\sqrt{d}} \\
&= N\cdot\frac{N}{d}\cdot\Big(\frac{1}{d}\sum_{i=1}^d\sigma^{(1)}_i\tau^{(1)}_i\Big)\cdots\Big(\frac{1}{d}\sum_{i=1}^d\sigma^{(p)}_i\tau^{(p)}_i\Big)
\stackref{overlap_def}{=} Np\prod_{s=1}^p R_s(\sigma,\tau).
}
On the other hand,
\begin{align*}
\text{RHS of \eqref{ug336}} 
&\stackrefpp{cov_a}{cov_b,73hb}{=} N\xi_N\big(\vc R(\sigma,\tau)\big) \\
&\stackref{cov_b,73hb}{=}N\sum_{\text{distinct $s_1,\dots,s_p$}}\frac{p^{p+1}}{p!}\cdot\Big(\frac{1}{p}\Big)^p\cdot R_{s_1}(\sigma,\tau)\cdots R_{s_p}(\sigma,\tau) \\
&\stackrefp{cov_b,73hb}{=}Np\prod_{s=1}^p R_s(\sigma,\tau).
\end{align*}
Comparing the two previous displays, we obtain \eqref{ug336}.
\end{proof}

\begin{proof}[Proof of Corollary~\ref{cor_inj_norm}]
Fix $p\ge3$ and adopt the setting of Lemma~\ref{lem_translate}: $N = pd$, $\SSS=\{1,\dots,p\}$, $\lambda_s = \lambda_{s,N} = \frac{1}{p}$ for each $s$, and $\vc\Delta$ is given by \eqref{73hb}.
For any $\eps>0$, the distributional equality \eqref{38gcb} implies
\eeq{ \label{hs8v3}
\P\Big(\frac{\|T\|_{\inj}}{\sqrt{d}} < \sqrt{p}E_0(p) - \eps\Big)
= \P\Big(\max_{\sigma\in\Sigma_{N,\sph}}\frac{H_N(\sigma)}{N} < \sqrt{p}E_0(p) - \eps\Big).
}
It is easy to see that the balanced condition \eqref{balanced} holds: use either Example~\ref{ex_model}\ref{ex_model_1} or Example~\ref{ex_model}\ref{ex_model_3}.
Therefore, we may apply Theorem~\ref{thm_main}, which gives
\eq{
\liminf_{N\to\infty}\GSE_{N,\sph}(\vc\Delta,\vc\lambda)
\ge \lim_{N\to\infty}\GSE_{N,\sph}^\single(\bbeta),
}
where $\bbeta$ is given by \eqref{beta_def}.
Because this $\bbeta$ is pure $p$-spin (i.e.\ $\beta_q = 0$ for all $q\neq p$) with $\beta_p^2 = \xi(1,\dots,1) \stackref{37gcb}{=} p$, the limit statement \eqref{3r8gx} guarantees
\eq{
\lim_{N\to\infty}\GSE_{N,\sph}^\single(\bbeta)
= \sqrt{p}E_0(p).
}
Combining the two previous displays, we obtain the following inequality for all large $N$:
\eeq{ \label{hs8v4}
&\P\Big(\max_{\sigma\in\Sigma_{N,\sph}}\frac{H_N(\sigma)}{N} < \sqrt{p}E_0(p) - \eps\Big) \\
&\le
\P\Big(\max_{\sigma\in\Sigma_{N,\sph}}\frac{H_N(\sigma)}{N} < \GSE_{N,\sph}(\vc\Delta,\vc\lambda) - \frac{\eps}{2}\Big)
\stackref{borell_tis}{\le}2\exp\Big(\frac{-N(\eps/2)^2}{2p}\Big).
}
Together, \eqref{hs8v3} and \eqref{hs8v4} imply
\begin{align*}
\limsup_{d\to\infty}\frac{1}{d} \log\P\Big(\frac{\|T\|_{\inj}}{\sqrt{d}} < \sqrt{p}E_0(p) - \eps\Big) 
&\le -\frac{\eps^2}{8}.
\intertext{On the other hand, it was shown in \cite[Theorem~1.1]{dartois_mckenna24_arxiv} that}
\limsup_{d\to\infty}\frac{1}{d}\log\P\Big(\frac{\|T\|_{\inj}}{\sqrt{d}} > \sqrt{p}E_0(p) + \eps\Big) 
&< 0.
\end{align*}
In light of the two previous displays, it follows from Borel--Cantelli (and sending $\eps\searrow0$) that
\[
\lim_{d\to\infty}\frac{\|T\|_\inj}{\sqrt{d}} = \sqrt{p}E_0(p)
\quad \mathrm{a.s.}
\qedhere
\]
\end{proof}

\section{Acknowledgments}
We thank Benjamin McKenna for bringing the work~\cite{dartois_mckenna24_arxiv} to our attention, which inspired Corollary~\ref{cor_inj_norm}. 
We are grateful to him and Stephane Dartois for discussion of their upcoming work \cite{dartois_mckenna25_prep} mentioned in Remark~\ref{rmk:DM25}, as well as several very helpful comments on a preliminary draft. We thank Victor Issa for explaining the connections with \cite{issa24_arxiv}, which are discussed in Remark~\ref{rem_compare} and highlighted elsewhere throughout the paper.
We thank Antonio Auffinger, Wei-Kuo Chen, Jean-Christophe Mourrat, and Dmitry Panchenko for discussions about the interpolation method of Proposition~\ref{prop_key}.
E.B. was partially supported by National Science Foundation grant DMS-2412473.

\bibliographystyle{myacm}
\bibliography{erikbib}{}

\end{document}